\documentclass[11pt, reqno]{amsart}

\usepackage{amsmath,amssymb,amsthm,mathtools}
\usepackage[hidelinks]{hyperref}
\usepackage{xcolor}
\usepackage{enumitem}
\usepackage{comment}

\hypersetup{colorlinks=true, linkcolor=blue, citecolor=blue, urlcolor=blue}
\numberwithin{equation}{section}
\newtheorem{theorem}{Theorem}[section]
\newtheorem{proposition}[theorem]{Proposition}
\newtheorem{lemma}[theorem]{Lemma}

\newtheorem{remark}[theorem]{Remark}

\newcommand{\Ft}{\mathbb F_2}
\newcommand{\Fp}{\mathbb F_p}
\newcommand{\C}{\mathbb C}
\newcommand{\Heis}{\mathsf H}
\newcommand{\Stief}[2]{\mathrm{St}(#1,#2)}
\newcommand{\Span}{\operatorname{span}}
\newcommand{\Ent}{\operatorname{Ent}}
\newcommand{\Var}{\operatorname{Var}}
\newcommand{\GLn}{\mathrm{GL}_n(\Ft)}
\newcommand{\TV}{\mathrm{TV}}
\newcommand{\Prob}{\mathbb P}
\newcommand{\E}{\mathbb E}
\newcommand{\cE}{\mathcal E}
\newcommand{\one}{\mathbf 1}
\newcommand{\supp}{\operatorname{supp}}
\newcommand{\wt}{\widetilde}
\newcommand{\Good}{\mathcal G}
\newcommand{\Unif}{\operatorname{Unif}}
\newcommand{\Irr}{\operatorname{Irr}}

\title[$k$-Column Transvections via Local Entropy]
{Mixing on $k$ Columns of the Transvection Walk}

\author{Natesh Pillai}
\address{Department of Statistics, Harvard University}

\author{Aaron Smith}
\address{Department of Mathematics and Statistics, University of Ottawa}

\date{}

\begin{document}

\maketitle

\begin{abstract}
In \cite{Diaconis1996WalksOG}, Diaconis and Saloff-Coste introduced the simple ``transvection" walk on $\mathrm{GL}_n(\mathbb F_2)$: at each
step, choose two distinct rows and add one to the other. In \cite{BenHamou2025}, Ben-Hamou recently proved that this walk has mixing time $O(n^2\log n)$. Inspired by applications in cryptography \cite{sotirakitrap}, the paper \cite{BenHamouPeres2018} conjectured that the first $k$ columns of this walk mixed in $O(nk \log(n))$ steps. Our main result is a proof of this conjecture uniformly in $n$ and $k.$

Our proof is based on a local-to-global entropy estimate, in the spirit of block factorization results such as
\cite{CaputoMenzTetali2015ApproxTensor,CaputoParisi2021BlockFactorization}. In our setting, the kernels that correspond roughly to the block kernels of \cite{CaputoParisi2021BlockFactorization} do not have uniformly large log-Sobolev constants, and so naively applying these techniques does not improve over \cite{BenHamou2025}. We avoid these bad blocks by combining our entropy estimates with a burn-in argument similar to classical drift-and-minorization arguments \cite{Rosenthal1995Minorization}. This method may be of broader interest, and so we illustrate it by proving an analogous result for a family of product-replacement algorithms on the Heisenberg group. 
\end{abstract}

\section{Introduction}\label{sec:introduction}

The \emph{transvection walk} on \(\GLn\), introduced in
\cite{Diaconis1996WalksOG}, is the Markov chain obtained from the following row
operation.  Given the current matrix \(\widehat Z_t\), sample distinct ``donor" and ``recipient" indices
\(a_t,b_t\in\{1,\ldots,n\}\) uniformly at random and set
\[
\widehat Z_{t+1}[\ell,*]
=
\begin{cases}
\widehat Z_t[b_t,*]+\widehat Z_t[a_t,*], & \ell=b_t,\\[4pt]
\widehat Z_t[\ell,*], & \ell\ne b_t.
\end{cases}
\]
 The upper bound
\(O(n^2\log n)\) for the mixing time of this walk is proved in \cite[Theorem~1]{BenHamou2025}. This is conjectured to be optimal. A simple diameter-based lower bound gives $\tau_{\mathrm{mix}} = \Omega \left( \frac{n^{2}}{\log(n)}\right)$, so Theorem 1 of \cite{BenHamou2025} is sharp up to $\text{polylog}(n)$ factors.

In several applications, including the trapdoored-matrix authentication proposal of \cite{sotirakitrap}, it is natural to only care about the mixing of a comparatively
small collection of columns. This motivates the present paper, where we study the projection of the transvection walk onto its first \(k\)
columns.  Writing \(V=\Ft^k\), an element of the state space is an \(n\)-tuple of row vectors whose
span is all of \(V\):
\[
\Stief{n}{k}
=\{(z_1,\ldots,z_n)\in V^n:\Span(z_1,\ldots,z_n)=V\}.
\]
The associated transition kernel \(P_{n,k}\) chooses an ordered donor-recipient
pair \((a,b)\), with \(a\ne b\), and replaces \(z_b\) by \(z_b+z_a\).  Its
stationary distribution is uniform on \(\Stief{n}{k}\).  To avoid periodicity
issues, we state the main result for the half-lazy kernel
\[
    Q_{n,k}=\frac12(I+P_{n,k}).
\]
Our main result settles a conjecture of \cite{BenHamouPeres2018}:

\begin{theorem}\label{thm:intro-transvection-main}
There is an absolute constant \(C<\infty\) such that, for all sufficiently large
\(n\) and every \(1\le k\le n\), the half-lazy \(k\)-column transvection walk on
\(\Stief{n}{k}\) satisfies
\[
    \tau_{\mathrm{mix}}\le C nk\log n .
\]
\end{theorem}

At the end of Section~\ref{sec:transvection}, we prove a lower bound that matches up to $polylog(n)$ factors:

\begin{theorem}\label{thm:intro-transvection-lower}
There is an absolute constant \(c>0\) such that, for all sufficiently large
\(n\) and every \(1\le k\le n\), the half-lazy \(k\)-column transvection walk on
\(\Stief{n}{k}\) satisfies
\[
    \tau_{\mathrm{mix}}
    \ge c\max\left(n\log n,\frac{nk}{\log(e n)}\right).
\]
\end{theorem}

The proof of Theorem \ref{thm:intro-transvection-main} begins by comparing the transvection walk to a walk on a slightly larger state space, as in the previous work \cite{BenHamou2025}. It also directly uses the main result of \cite{BenHamou2025} to deal with the case that $k \approx n$ is large. Our arguments are otherwise  essentially unrelated to those of \cite{BenHamou2025}. 

The main idea behind our proof is to compare the transvection walk to a simpler collection of ``frozen" walks obtained by conditioning on both the recipient coordinate and all non-recipient rows. The proof requires three main steps. First, we show that these frozen walks mix rapidly for a set of ``typical" collections of donor vectors. Second, we show that the Markov chain rapidly enters this set of typical donor vectors and remains in this set for a long time. Finally, we show that it is possible to combine local estimates from steps 1 and 2 into a global bound.  We think that this third step might be of independent interest, as it is often very natural to try to bound the mixing time of a complicated Markov chain in terms of simpler frozen or ``censored" chains and these comparisons are known to be somewhat difficult in general \cite{PeresWinkler2013Censoring,FillKahn2013Comparison}.

In order to illustrate the method, in Section \ref{sec:heis} we introduce and analyze a random walk on generators of the Heisenberg group that is similar to (but more complicated than) the transvection walk. 

\subsection{Notation and Detailed Proof sketch}\label{subsec:intro-proof-idea}

\subsubsection{Notation for Fibres and Good Sets} \label{SecFibreNotationFirstTime}

We set some notation related to product spaces, fibres, and ``good" sets that will be used throughout the proof sketch and the rest of the paper. Since our results will apply to walks other than the transvection walk, we use more generic notation.

For a product space \(X=S^N\), 
\(x=(x_1,\ldots,x_N)\in X\), \(i\in[N]\) and \(u\in S\), write
\[
x^{i,u}
=
(x_1,\ldots,x_{i-1},u,x_{i+1},\ldots,x_N).
\]
For a fixed index $i$ and configuration $x \in S^{N}$, define
\[
x_{-i}=(x_1,\ldots,x_{i-1},x_{i+1},\ldots,x_N)\in S^{N-1}
\]
and then define the corresponding \(i\)-fibre 
\[
\mathcal F_i(x_{-i})
=
\{x^{i,u}:u\in S\}\subseteq X.
\]
For \(G\subseteq X\), we call this fibre \(G\)-good if
\[
\mathcal F_i(x_{-i})\cap G\ne\varnothing.
\]
When the set $G$ is clear from the context, we simply call the fibre \textit{good.}

\begin{remark} [Relationship to Theorem \ref{thm:intro-transvection-main}]
 In the \(k\)-column transvection walk, we set \(S=V=\mathbb F_2^k\), \(N=n\),
and \(x^{i,u}\) is the row tuple obtained from \(x\) by replacing row \(i\)
by \(u\). Notice that these fibres are fibres in the ambient product space; they are generally not contained in the subset \(\Stief{n}{k}\).
\end{remark}

\subsubsection{Detailed Proof Sketch for the Transvection Walk}
We give a proof sketch for the walk on \(\GLn\); the basic approach for the walk on the Heisenberg group will be similar. We use notation from this section when analyzing the transvection walk in Section \ref{sec:transvection}.

Let \(V=\Ft^k\) and let $\mu$ be the uniform distribution on $V^{n}$. For \(\xi\in V^*\), define the associated function $\chi_\xi \, : \, V \to \{-1,1\}$ by the formula
\[
\chi_\xi(u)=(-1)^{\xi(u)}, \qquad u\in V.
\]
We extend this to a function $S_\xi \, : \, V^{n} \to \mathbb{Z}$ by the formula
\[
S_\xi(z)=\sum_{i=1}^n \chi_\xi(z_i).
\]
The first important step is to show that the restriction of the walk to functions on the
following set of ``good" or ``well-balanced" elements of  \(\Stief{n}{k}\):
\begin{equation} \label{DefSketchFourierGood}
\Good_{\mathrm{tr}}
=\left\{z\in\Stief{n}{k}:\max_{0\ne\xi\in V^*}|S_\xi(z)|\le\frac n4\right\}
\end{equation}
satisfies a log-Sobolev inequality with constant $\Omega (\frac{1}{nk})$. See Inequality \eqref{IneqLocalLogSob}.

We now describe the proof of Theorem~\ref{thm:tr-linear-lazy}. Fix a recipient row \(i\) and condition on the values $\{z_{j}\}_{j \neq i}$ of all
other rows. Conditioning one step of the walk on row \(i\) being the recipient and then projecting to the fibre gives the following kernel on the possible value \(u=z_i\in V\):
\[
    K_i\varphi(u)=\frac1{n-1}\sum_{j\ne i}\varphi(u+z_j).
\]
The functions \( \{ \chi_\xi \}_{\xi\in V^*}\) diagonalize \(K_i\), and the eigenvalue associated with index \(\xi\ne0\) is
\[
    \frac1{n-1}\sum_{j\ne i}(-1)^{\xi(z_j)}.
\]
By Proposition \ref{prop:tr-fibre-gap}, the condition that the fibre $\mathcal F_i(z_{-i})$ be \(\Good_{\mathrm{tr}}\)-good is equivalent, up to a contribution that is negligible for large $n$, to a uniform bound on these nontrivial eigenvalues.

In particular, the operators associated with good fibres have spectral gaps uniformly bounded away from 0; see Proposition~\ref{prop:tr-fibre-gap}. The tensorization property of entropy (see Inequality \eqref{eq:abstract-tensorization}), combined with a standard bound on log-Sobolev inequalities in terms of spectral gaps, then gives a local estimate through Proposition~\ref{prop:abstract-fibre-gap-lsi}
\begin{equation} \label{IneqSketchLSI}
\Ent_\mu(F^2)\le C nk\,\cE_{P_{n,k}}^\mu(F,F)
\end{equation}
for functions $F$ supported on \(\Good_{\mathrm{tr}}\).

We will see that, in order to use these local log-Sobolev inequalities to obtain a global bound, it is enough to show that the chain enters \(\Good_{\mathrm{tr}}\) quickly and remains there for a long time; see Lemma~\ref{lem:tr-burnin} and Theorem~\ref{thm:abstract-good-set}. The required burn-in estimates can be obtained by careful use of existing strong results on the mixing time of the 1-column walk \cite[Theorem~1]{BenHamouPeres2018}. More precisely, for each fixed \(\xi\ne0\), the projection
\[
(z_1,\ldots,z_n)\mapsto(\xi(z_1),\ldots,\xi(z_n))
\]
is simply a copy of the one-column transvection walk on \(\Ft^n\setminus\{0\}\).
The one-column estimates \cite[Theorem~1]{BenHamouPeres2018} and
\cite[Theorem~1.1]{PeresTanakaZhai2020}, combined with a union bound over the \(2^k-1\) nonzero
functionals, imply that the \(k\)-column walk enters \(\Good_{\mathrm{tr}}\)
after \(O(nk\log n)\) steps and then remains there for polynomially-in-$n$ many steps
with high probability; see Lemma~\ref{lem:tr-burnin}.

Putting these ideas together: we show that the chain enters a good set quickly, prove a log-Sobolev inequality like \eqref{IneqSketchLSI} on this good set, and finally show that mixing on the good set implies fast mixing on the full state space by Theorem~\ref{thm:abstract-good-set}.

Although this proof sketch is stated for the transvection walk, we will see in Section~\ref{sec:heis} that the same three-step structure of burn-in, local fibre estimates, and patching is used for the Heisenberg walk.  The main differences come from the analogues of the functions $S_{\xi}$: for the Heisenberg walk these are replaced by the irreducible representations described in Proposition~\ref{prop:heis-reps} and used in Proposition~\ref{prop:heis-fibre-gap}.

\subsection{Related work}\label{subsec:related-work}

There is a long history of studying the mixing times of random walks on groups; see, for example, \cite{Diaconis1988,SaloffCoste2004,levin2017markov}. The product-replacement algorithm, introduced in
\cite{CellerLeedhamGreenMurrayNiemeyerOBrien1995}, is one of the most computationally-relevant such walks, as it is an efficient algorithm for generating random group elements; see \cite{Pak2001} for
background. Closely-related walks on generating sets of abelian groups were studied in
\cite{Diaconis1996WalksOG}, and their Theorem 4.1 gives a mixing time bound for the transvection walk studied in this paper.  

The transvection walk is one of the better-studied special cases of the product-replacement algorithm. The spectral gap was computed as a consequence of Kassabov's Kazhdan-constant estimates \cite{Kassabov2005}. Since then, substantial progress has been made. The one-column walk is closely related to the stratified hypercube walk studied in \cite{ChungGraham1997} and was analyzed sharply in \cite{BenHamouPeres2018}. The paper \cite{PeresTanakaZhai2020} showed cutoff for product replacement algorithms on fixed finite groups, and in particular showed cutoff for the $k$-column walk for all fixed $k$ as $n$ grows. For the full transvection walk, the correct mixing time was found in \cite[Theorem~1]{BenHamou2025}.  

Several closely-related random transvection walks on \(\mathrm{SL}_n(\mathbb F_q)\) were studied by Hildebrand \cite{Hildebrand1992}. The Heisenberg walk example in this paper is not as well-studied. We introduce it in this paper as a test case for our method and to show how one might deal with a problem that occurs in most walks but not in the transvection walk: the fibre kernel $K_i$ can no longer be diagonalized in an extremely simple form.

At a high level, the ingredients in our paper are  standard:
burn-in estimates, local Poincar\'e/log-Sobolev inequalities and a tensorization inequality, and a patching argument.  The main contribution is making them work together for these walks.  

The use of local log-Sobolev inequalities and the tensorization inequality \eqref{eq:abstract-tensorization} is closest in spirit to entropy
factorization methods as in 
\cite{Cesi2001QuasiFactorization,LuYau1993KawasakiGlauber,
CaputoMenzTetali2015ApproxTensor,CaputoParisi2021BlockFactorization,
BlancaCaputoChenParisiStefankovicVigoda2022Mixing}.  These methods are also somewhat related to decomposition methods for Markov chains
\cite{MadrasRandall2002Decomposition,JerrumSonTetaliVigoda2004Decomposable}. The local log-Sobolev inequalities are obtained for what we call ``frozen" kernels, obtained by conditioning on or ``freezing" certain parts of the state space. This is similar in spirit to previous work on censorship inequalities
\cite{PeresWinkler2013Censoring,FillKahn2013Comparison}, though we do not have any similar monotonicity result. 

Finally, we find that the frozen kernels do not all have large spectral gaps. We avoid the kernels with small spectral gaps by using a burn-in argument. This is similar in spirit to the classical drift-and-minorization argument, which proves rapid mixing on a good set with a simple coupling bound and uses a burn-in argument to deal with the remainder of the state space \cite{Rosenthal1995Minorization}. The same basic approach of combining local estimates with a burn-in estimate has been used with other functional inequalities in many other contexts, as in \textit{e.g.} \cite{chen2020fast}. The details of the burn-in argument are idiosyncratic to these particular walks.

\subsection{Paper Outline}
Section~\ref{sec:framework} states and proves a generic mixing time upper bound, Theorem~\ref{thm:abstract-good-set}, for a class of random walks that includes the product-replacement algorithms.  Section~\ref{sec:transvection} verifies the assumptions of Theorem~\ref{thm:abstract-good-set} for the \(k\)-column transvection walk, proving our main result
Theorem~\ref{thm:intro-transvection-main} as a consequence.  Section~\ref{sec:heis} introduces a version of the product-replacement algorithm for the Heisenberg group and proves an analogous result, Theorem~\ref{thm:heis-main}, as a consequence. It also proves a nearly-matching lower bound, Theorem~\ref{thm:heis-lower}. 

\section{Generic Bounds for ``Coordinate Replacement" Walks}\label{sec:framework}

This section is written independently of the two main examples.

\subsection{Generic notation}

For a finite set $S$, we write $u_{S}$ for the uniform distribution on $S$.

For a probability measure \(\rho\) on a finite set \(E\), write
\[
    \Ent_\rho(f^2)=\rho(f^2\log f^2)-\rho(f^2)\log\rho(f^2).
\]
If \(K\) is a reversible Markov kernel with stationary measure \(\rho\), define the Dirichlet form
\[
    \cE_K^\rho(f,f)=\frac12\sum_{x,y\in E}\rho(x)K(x,y)(f(x)-f(y))^2,
\]
with $\cE_K^\rho(f,g)$ the usual bilinear extension. When \(\rho\) is clear from context, we suppress it from the notation.

We often switch between continuous and discrete time for convenience. When bounds apply in both situations without important change, we sometimes replace a time $t$ by $\lceil t \rceil$ without comment.

A substochastic kernel is a nonnegative kernel whose row sums are at most one. If
\(K\) is a reversible substochastic kernel on \((E,\rho)\), we use the
equivalent form
\[
    \cE_K^\rho(f,f)=\langle f,(I-K)f\rangle_\rho.
\]

For a reversible Markov kernel \(K\) with stationary law \(\rho\), we define the
Poincar\'e constant \(C_P\) to be the smallest constant $C$ satisfying
\[
\Var_\rho(f)\le C\,\cE_K^\rho(f,f)
\]
for all real \(f\). The spectral gap is the reciprocal of this constant.

If \(K\) is a kernel on a finite set \(E\) and \(A\subseteq E\) is invariant under \(K\), then the restriction of \(K\) to \(A\) is the Markov kernel \(K_A(x,y)=K(x,y)\), \(x,y\in A\).

For a nonnegative function \(u\) on \(E\), define the entropy of the
subprobability density \(u\) by
\[
    H_\rho(u)=\rho\left[u\log\frac{u}{\rho(u)}\right],
\]
with the usual convention at \(u=0\).

We recall the following result for obtaining log-Sobolev inequalities from Poincar\'e inequalities (see \textit{e.g.} \cite{DiaconisSaloffCoste1996LSI}):

\begin{lemma}[Poincar\'e to log-Sobolev]\label{lem:abstract-gap-lsi} There is a universal constant $C < \infty$ such that the following holds.
Let \(K\) be reversible on a finite set \(E\), with stationary law \(\rho\) and
Poincar\'e constant at most \(C_P\).  Then for every real \(f\),
\[
    \Ent_\rho(f^2)\le C C_P\log(1/\rho_*)\,\cE_K^\rho(f,f),
\]
where $\rho_*=\min_{x\in E}\rho(x)$.
\end{lemma}

\subsection{Coordinate-replacement kernels}\label{subsec:coordinate-replacement}

Let \(S\) be a finite set, let $N \in \mathbb{N}$, let $\mathcal X=S^N$, and let $\mu$ be the uniform measure on $\mathcal X$. We use the fibre notation introduced in Section \ref{SecFibreNotationFirstTime}.

We consider kernels \(P\) on \(\mathcal X\) with the following structure.  For
each coordinate \(i\), and each
\(x_{-i}=(x_1,\ldots,x_{i-1},x_{i+1},\ldots,x_N)\), there is a Markov kernel
\(K_i(x_{-i})\) on \(S\), reversible with respect to uniform measure on \(S\),
such that
\begin{equation}\label{eq:abstract-fibre-decomposition}
    Pf(x)=\frac1N\sum_{i=1}^N
    \sum_{u\in S}K_i(x_{-i};x_i,u)f(x^{i,u}),
\end{equation}
where \(x^{i,u}\) is obtained from \(x\) by replacing coordinate \(i\) by
\(u\).  We call such kernels \textit{coordinate replacement} kernels. The walks studied in this paper do not have this form, but the walks obtained by conditioning on the recipient coordinate do have this form; their fibre kernels are displayed in \eqref{eq:tr-fibre-kernel} and \eqref{eq:heis-fibre-kernel}.

Let \(\Omega\subseteq\mathcal X\) be a set with $\mu(\Omega) > 0$ that is invariant under $P$, and define the restricted measure
\[
\pi=\mu(\cdot\mid\Omega).
\]
We also assume that \(P\)  is reversible with respect to \(\mu\) on \(\mathcal X\); hence the restriction of $P$ to \(\Omega\) is reversible with respect to \(\pi\). 

We shall repeatedly use the tensorization inequality
\begin{equation}\label{eq:abstract-tensorization}
    \Ent_\mu(F^2)
    \le
    \sum_{i=1}^N
    \mu\left[\Ent_\mu\bigl(F^2\mid (X_\ell)_{\ell\ne i}\bigr)\right],
\end{equation}
where $\mu$ is a product measure and $X \sim \mu$.

We apply Lemma \ref{lem:abstract-gap-lsi} to our coordinate-replacement kernels:

\begin{proposition}[Fibre-gap to local-LSI ]\label{prop:abstract-fibre-gap-lsi}
Let \(P\) be a coordinate-replacement kernel of the form
\eqref{eq:abstract-fibre-decomposition}, and let \(G\subseteq\Omega\).  Suppose
there is \(\gamma>0\) such that, for every coordinate \(i\) and every
configuration \(x_{-i}\) whose \(i\)-fibre
\[
    \{x^{i,u}:u\in S\}
\]
intersects \(G\), the fibre kernel \(K_i(x_{-i})\) has spectral gap at least
\(\gamma\) with respect to uniform measure on \(S\).  Then every nonnegative
\(F:\mathcal X\to\mathbb R\) with \(\supp(F)\subseteq G\) satisfies
\begin{equation}\label{eq:abstract-local-lsi}
    \Ent_\mu(F^2)
    \le
    C\frac{N\log |S|}{\gamma}\,\cE_P^\mu(F,F),
\end{equation}
where \(C<\infty\) is universal.
\end{proposition}

\begin{proof}
We begin with the tensorization inequality \eqref{eq:abstract-tensorization}.  Fix \(i\) and condition
on \(X_{-i}=x_{-i}\).  If the fibre does not intersect \(G\), then \(F\)
vanishes identically on that fibre and the conditional entropy is zero.  If the
fibre intersects \(G\), apply Lemma~\ref{lem:abstract-gap-lsi} to the kernel
\(K_i(x_{-i})\), whose stationary law is uniform on \(S\), to get
\[
    \Ent_{u\sim\Unif(S)}(F(x^{i,u})^2)
    \le C\frac{\log |S|}{\gamma}
    \cE_{K_i(x_{-i})}(u\mapsto F(x^{i,u}),u\mapsto F(x^{i,u})).
\]
Averaging over \(x_{-i}\) and summing over \(i\), while using
\eqref{eq:abstract-fibre-decomposition}, gives
\[
    \sum_{i=1}^N\mu\bigl[\cE_{K_i(X_{-i})}(F,F)\bigr]
    =N\cE_P^\mu(F,F),
\]
which proves \eqref{eq:abstract-local-lsi}.
\end{proof}

\subsection{Entropy decay for killed chains}\label{subsec:killed-entropy}

We will be proving log-Sobolev inequalities only on ``good" sets such as that in Equation \eqref{DefSketchFourierGood}. In order to apply these to the full state space, we will first apply the log-Sobolev inequalities to a sub-stochastic chain that is ``killed" when it leaves the good set, then show that paths of the original chain are close (in an appropriate sense) to those in the killed chain. This section collects the estimates required to apply log-Sobolev inequalities to sub-stochastic chains, and mimics standard estimates such as those in \cite{DiaconisSaloffCoste1996LSI} quite closely.

Let \(G\subseteq\Omega\) with \(\pi(G)>0\), and let \(K_G\) be \(P\) killed on
exiting \(G\):
\begin{equation}\label{eq:abstract-killed-kernel}
    K_G(x,y)=P(x,y)\one_{\{x,y\in G\}},
    \qquad x,y\in G.
\end{equation}
Let \(\pi_G=\pi(\cdot\mid G)\).  The kernel \(K_G\) is reversible and
substochastic on \((G,\pi_G)\).  Its average killing rate is
\[
    \delta_G=\pi_G(1-K_G1).
\]
By reversibility of \(P\),
\begin{equation}\label{eq:abstract-killing-rate}
    \delta_G\le\frac{\pi(G^c)}{\pi(G)}.
\end{equation}

\begin{proposition}\label{prop:abstract-zero-extension}
Assume that, for every nonnegative \(F:\mathcal X\to\mathbb R\) with
\(\supp(F)\subseteq G\),
\begin{equation}\label{eq:abstract-ambient-lsi}
    \Ent_\mu(F^2)\le A\cE_P^\mu(F,F).
\end{equation}
Then every nonnegative \(f:G\to\mathbb R\) satisfies
\begin{equation}\label{eq:abstract-killed-lsi}
    \Ent_{\pi_G}(f^2)\le A\cE_{K_G}^{\pi_G}(f,f).
\end{equation}
\end{proposition}

\begin{proof}
Extend \(f\) by zero to \(\mathcal X\):
\[
    \wt f(x)=f(x)\one_G(x).
\]
Since \(G\subseteq\Omega\) and \(\pi=\mu(\cdot\mid\Omega)\), we have
\(\mu(\cdot\mid G)=\pi_G\).  Put \(q=\mu(G)\).  Then
\[
    \Ent_\mu(\wt f^2)
    =q\Ent_{\pi_G}(f^2)+q\pi_G(f^2)\log\frac1q
    \ge q\Ent_{\pi_G}(f^2).
\]
Also
\[
\begin{aligned}
    \cE_P^\mu(\wt f,\wt f)
    &=\langle \wt f,(I-P)\wt f\rangle_\mu  \\
    &=q\langle f,(I-K_G)f\rangle_{\pi_G}
     =q\cE_{K_G}^{\pi_G}(f,f).
\end{aligned}
\]
Applying \eqref{eq:abstract-ambient-lsi} to \(\wt f\) and dividing by \(q\)
proves the claim.
\end{proof}

\begin{lemma}[Entropy decay for a substochastic chain]\label{lem:abstract-killed-entropy}
Let \(K\) be a reversible substochastic kernel on a finite probability space
\((E,\rho)\), and let
\[
    \delta=\rho(1-K1).
\]
Assume that, for all nonnegative \(f\),
\begin{equation}\label{eq:abstract-substoch-lsi}
    \Ent_\rho(f^2)\le A\cE_K^\rho(f,f).
\end{equation}
Let \(T_t=e^{t(K-I)}\) and set \(u_t=T_tu_0\) for any \(u_0\ge0\). Then
\begin{equation}\label{eq:abstract-killed-entropy-decay}
    H_\rho(u_t)
    \le
    e^{-t/A}H_\rho(u_0)
    +A\delta\,\rho(u_0)(1-e^{-t/A}).
\end{equation}
\end{lemma}

\begin{proof}
It is enough to consider only strictly positive \(u_t\).  Let 
\(m_t=\rho(u_t)\) and \(h_t=\sqrt{u_t}\).  Differentiating gives
\[
    \frac{d}{dt}H_\rho(u_t)
    =-\cE_K^\rho\left(u_t,\log\frac{u_t}{m_t}\right).
\]
Write
\[
    w_{xy}=\rho(x)K(x,y),
    \qquad
    \kappa_x=\rho(x)(1-K1(x)).
\]
Then, for real \(f,g\),
\begin{equation}\label{eq:abstract-substoch-bilinear}
    \cE_K^\rho(f,g)
    =
    \frac12\sum_{x,y}w_{xy}(f(x)-f(y))(g(x)-g(y))
    +\sum_x\kappa_x f(x)g(x).
\end{equation}
We claim that
\begin{equation} \label{SubclaimKilledProbsRatio}
    \cE_K^\rho\left(u_t,\log\frac{u_t}{m_t}\right)
    \ge
    \cE_K^\rho(h_t,h_t)-m_t\delta.
\end{equation}
Write \(u_x=u_t(x)\), \(h_x=h_t(x)=\sqrt{u_x}\), and \(m=m_t\).  For the
first term in \eqref{eq:abstract-substoch-bilinear},
\[
\begin{aligned}
&\frac12\sum_{x,y}w_{xy}(u_x-u_y)
\left(\log\frac{u_x}{m}-\log\frac{u_y}{m}\right)  \\
&\qquad =
\frac12\sum_{x,y}w_{xy}(u_x-u_y)(\log u_x-\log u_y)  \\
&\qquad \ge
2\sum_{x,y}w_{xy}(h_x-h_y)^2
\ge
\frac12\sum_{x,y}w_{xy}(h_x-h_y)^2,
\end{aligned}
\]
where we used
\[
    (a-b)(\log a-\log b)\ge 4(\sqrt a-\sqrt b)^2.
\]
For the second term in \eqref{eq:abstract-substoch-bilinear}, applying
\(s\log s\ge s-1\) with \(s=u_x/m\) gives
\[
    u_x\log\frac{u_x}{m}
    =
    m\left(\frac{u_x}{m}\log\frac{u_x}{m}\right)
    \ge
    u_x-m.
\]
Therefore
\[
\begin{aligned}
    \sum_x\kappa_x u_x\log\frac{u_x}{m}
    &\ge
    \sum_x\kappa_x(u_x-m)  \\
    &=
    \sum_x\kappa_x h_x^2
    -
    m\sum_x\kappa_x  \\
    &=
    \sum_x\kappa_x h_x^2
    -
    m\delta,
\end{aligned}
\]
since
\[
    \sum_x\kappa_x
    =
    \sum_x\rho(x)(1-K1(x))
    =
    \rho(1-K1)
    =
    \delta.
\]
Combining the last two estimates in \eqref{eq:abstract-substoch-bilinear}
gives
\[
    \cE_K^\rho\left(u_t,\log\frac{u_t}{m_t}\right)
    \ge
    \frac12\sum_{x,y}w_{xy}(h_t(x)-h_t(y))^2
    +\sum_x\kappa_xh_t(x)^2
    -m_t\delta
    =
    \cE_K^\rho(h_t,h_t)-m_t\delta,
\]
proving Inequality \eqref{SubclaimKilledProbsRatio}.

Continuing, by \eqref{eq:abstract-substoch-lsi},
\(H_\rho(u_t)=\Ent_\rho(h_t^2)\le A\cE_K^\rho(h_t,h_t)\).  Therefore
\[
    \frac{d}{dt}H_\rho(u_t)
    \le
    -\frac1A H_\rho(u_t)+m_t\delta.
\]
Since \(m_t\) is nonincreasing, \(m_t\le m_0=\rho(u_0)\).  Gronwall's
inequality proves \eqref{eq:abstract-killed-entropy-decay}.
\end{proof}

Finally, we include the following standard estimate for future reference:

\begin{lemma}\label{lem:abstract-subprob-tv}
Let \(\lambda=u\rho\) be a subprobability on \(E\), and set
\(m=\lambda(E)=\rho(u)\).  Then
\[
    \|\lambda-\rho\|_{\TV}
    \le
    (1-m)+\sqrt{\frac{H_\rho(u)}2}.
\]
\end{lemma}

\begin{proof}
If \(m=0\), the claim is trivial.  Otherwise write
\(\widehat\lambda=\lambda/m\).  Then
\[
    \|\lambda-\rho\|_{\TV}
    \le
    (1-m)+m\|\widehat\lambda-\rho\|_{\TV}.
\]
Pinsker's inequality gives
\[
    \|\widehat\lambda-\rho\|_{\TV}
    \le \sqrt{D(\widehat\lambda\Vert\rho)/2},
\]
and \(H_\rho(u)=mD(\widehat\lambda\Vert\rho)\).  Since \(m\le1\), the result
follows.
\end{proof}

\subsection{Mixing condition}\label{subsec:abstract-mixing}

We keep notation from the previous subsection. This subsection compares the original chain after burn-in with the chain killed on leaving the good set \(G\). The comparison lets entropy decay for the killed chain be transferred back to the original chain, provided paths rarely leave \(G\).

For $s \in \mathbb{N}$ and \(x\in\Omega\), define
\[
    \alpha_s^x=\delta_xP^s,
    \qquad
    \alpha_{s,G}^x(A)=\alpha_s^x(A\cap G).
\]
Let
\[
    \widetilde\alpha_{s,t}^x=\alpha_s^x e^{t(P-I)},
    \qquad
    \lambda_{s,t}^x=\alpha_{s,G}^x e^{t(K_G-I)}.
\]
Thus \(\widetilde\alpha_{s,t}^x\) is the original continuous-time evolution
after a discrete burn-in, while \(\lambda_{s,t}^x\) is the killed evolution
started from the restriction of \(\alpha_s^x\) to \(G\). In the following, $\operatorname{Poi}(t)$ denotes the Poisson distribution with rate $t$.

\begin{lemma}\label{lem:abstract-path-comparison}
For every integer \(s\ge0\), integer \(L\ge1\), and \(t\ge0\),
\begin{equation}\label{eq:abstract-path-comparison}
    \|\widetilde\alpha_{s,t}^x-\lambda_{s,t}^x\|_{\TV}
    \le
    \Prob_x(\exists\,u\in\{0,1,\ldots,L\}:X_{s+u}\notin G)
    +\Prob(\operatorname{Poi}(t)>L).
\end{equation}
The same right-hand side also bounds \(1-\lambda_{s,t}^x(G)\).
\end{lemma}

\begin{proof}
This is a simple coupling argument. Note that the number of jumps made in any time interval by the original chain is a Poisson process. Conditional on making \(r\) jumps
after time \(s\), the original and killed chains agree on every trajectory that
stays in \(G\) during the discrete interval \(s,s+1,\ldots,s+r\).  For
\(r\le L\), disagreement is contained in the event displayed in
\eqref{eq:abstract-path-comparison}; for \(r>L\), use the trivial bound \(1\).
Averaging over \(r\sim\operatorname{Poi}(t)\) gives the total variation bound,
and the same coupling gives the bound on the lost mass of the killed chain.
\end{proof}

We now have the main generic result. We restate some of the basic assumptions in the theorem statement for ease of reference:

\begin{theorem}\label{thm:abstract-good-set}
Let \(P\) be a coordinate-replacement kernel as in
\eqref{eq:abstract-fibre-decomposition} on \(\mathcal X=S^N\), with invariant
set \(\Omega\), product law \(\mu\), stationary law \(\pi=\mu(\cdot\mid\Omega)\),
and good set \(G\subseteq\Omega\).  Assume \(\pi(G)>0\).  Suppose that the
following hold.
\begin{enumerate}[label=\textup{(\roman*)}, leftmargin=*]
    \item There is a constant \(A<\infty\) such that, for every nonnegative
    \(F:\mathcal X\to\mathbb R\) with \(\supp(F)\subseteq G\),
    \[
        \Ent_\mu(F^2)\le A\cE_P^\mu(F,F).
    \]
    \item There are \(t_*\ge0\), an integer \(L\ge1\), and \(\eta\ge0\) such
    that, uniformly in \(x\in\Omega\) and all integers \(s\ge t_*\),
    \[
        \Prob_x(\exists\,u\in\{0,1,\ldots,L\}:X_{s+u}\notin G)\le\eta.
    \]
\end{enumerate}
Set
\begin{equation}\label{eq:abstract-tconf}
    t_{\mathrm{conf}}=2A\log(e+\log|\Omega|),
\end{equation}
and define
\begin{equation}\label{eq:abstract-error-defs}
    \zeta=\Prob(\operatorname{Poi}(t_{\mathrm{conf}})>L),
    \qquad
    R=e^{-t_{\mathrm{conf}}/A}\log|\Omega|+A\frac{\pi(G^c)}{\pi(G)}.
\end{equation}
Then, for every \(x\in\Omega\) and every integer \(s\ge t_*\),
\begin{equation}\label{eq:abstract-after-burnin}
    \left\|\delta_xP^s e^{t_{\mathrm{conf}}(P-I)}-\pi\right\|_{\TV}
    \le 2(\eta+\zeta)+\sqrt{R/2}+\pi(G^c).
\end{equation}
Consequently, if
\begin{equation}\label{eq:abstract-continuous-condition}
    2(\eta+\zeta)+\sqrt{R/2}+\pi(G^c)
    +\Prob(\operatorname{Poi}(2t_*)<t_*)\le\frac14,
\end{equation}
then the continuous-time chain satisfies
\[
    t_{\mathrm{mix}}^{\mathrm{cont}}(1/4)\le 2t_*+t_{\mathrm{conf}}.
\]
\end{theorem}

\begin{proof}
Write \(\pi_G=\pi(\cdot\mid G)\).  By
Proposition~\ref{prop:abstract-zero-extension}, applied to \(P\) and \(G\),
the killed kernel satisfies
\[
    \Ent_{\pi_G}(f^2)\le A\cE_{K_G}^{\pi_G}(f,f).
\]
Let \(\lambda_0\) be a subprobability on \(G\), let
\(u_0=d\lambda_0/d\pi_G\), and let \(u_t=e^{t(K_G-I)}u_0\).  If
\(m=\pi_G(u_0)>0\), then
\[
    H_{\pi_G}(u_0)=mD(\lambda_0/m\Vert\pi_G)
    \le \log |G|\le\log|\Omega|.
\]
If \(m=0\), the entropy is zero.  Lemma~\ref{lem:abstract-killed-entropy} and
\eqref{eq:abstract-killing-rate} give
\[
    H_{\pi_G}(u_t)
    \le e^{-t/A}\log|\Omega|+A\frac{\pi(G^c)}{\pi(G)}.
\]
At \(t=t_{\mathrm{conf}}\), the right-hand side is \(R\).

Now take \(\lambda_0=\alpha_{s,G}^x\).  Lemma~\ref{lem:abstract-path-comparison},
applied with \(t=t_{\mathrm{conf}}\), gives
\[
    \|\widetilde\alpha_{s,t_{\mathrm{conf}}}^x
      -\lambda_{s,t_{\mathrm{conf}}}^x\|_{\TV}\le \eta+\zeta,
    \qquad
    1-\lambda_{s,t_{\mathrm{conf}}}^x(G)\le\eta+\zeta.
\]
Combining the entropy bound with Lemma~\ref{lem:abstract-subprob-tv} gives
\[
    \|\lambda_{s,t_{\mathrm{conf}}}^x-\pi_G\|_{\TV}
    \le \eta+\zeta+\sqrt{R/2}.
\]
Finally, \(\|\pi_G-\pi\|_{\TV}=\pi(G^c)\).  The triangle inequality proves
\eqref{eq:abstract-after-burnin}.

For continuous time, let \(J\sim\operatorname{Poi}(2t_*)\).  Poissonization gives
\[
    \delta_x e^{(2t_*+t_{\mathrm{conf}})(P-I)}
    =\E\left[\delta_xP^J e^{t_{\mathrm{conf}}(P-I)}\right].
\]
On the event \(\{J\ge t_*\}\), use \eqref{eq:abstract-after-burnin}; on the complement,
use the trivial bound.  Condition~\eqref{eq:abstract-continuous-condition}
proves the continuous-time bound.
\end{proof}

\section{The \texorpdfstring{$k$}{k}-column transvection walk}\label{sec:transvection}

\subsection{Basic notation and facts about the transvection chain}\label{subsec:transvection-chain}
We recall some notation from the introduction, then introduce notation that will be used in the proofs. Fix \(1\le k\le n\) and set
\[
    V=\Ft^k,
    \qquad
    \Stief{n}{k}=\{z=(z_1,\ldots,z_n)\in V^n:\Span(z_1,\ldots,z_n)=V\}.
\]
Let \(\pi=\pi_{n,k}\) be uniform on \(\Stief{n}{k}\), and let \(\mu\) be
uniform on \(V^n\).  For ordered distinct indices \(a,b\in[n]\), define
\[
    z^{a\to b}_\ell=
    \begin{cases}
    z_b+z_a, & \ell=b,\\
    z_\ell, & \ell\ne b.
    \end{cases}
\]
We call \(a\) the donor and \(b\) the recipient. We define the kernel for the transvection walk
\begin{equation}\label{eq:tr-kernel}
    P=P_{n,k},
    \qquad
    Pf(z)=\frac1{n(n-1)}\sum_{a\ne b}f(z^{a\to b}).
\end{equation}
The maps \(z\mapsto z^{a\to b}\) are \(\mu\)-preserving involutions; hence \(P\)
is reversible with respect to \(\mu\) on \(V^n\) and with respect to \(\pi\) on
\(\Stief{n}{k}\). In particular, row addition preserves the span of the rows, so it preserves \(\Stief{n}{k}\). 

We also claim the chain is irreducible on $\Stief{n}{k}$. To see this, note that over $\Ft$ you can write the row-swap elementary row operation $(r_{i}, r_{j}) \to (r_{j}, r_{i})$ as a sequence of the three additions $(r_{i},r_{j}) \to (r_{i} + r_{j}, r_{j}) \to (r_{i} + r_{j}, r_{i}) \to (r_{j},r_{i})$; this means that all elementary row operations can be written in terms of the moves of this Markov chain.

In the notation of Section~\ref{sec:framework}, the
coordinate space is \(S=V\), the number of coordinates is \(N=n\), and the
recipient-coordinate fibre kernel is obtained by conditioning on the recipient coordinate and projecting the update to that fibre:
\begin{equation}\label{eq:tr-fibre-kernel}
    K_i\varphi(u)=\frac1{n-1}\sum_{j\ne i}\varphi(u+z_j).
\end{equation}

\subsection{Description of the good set}\label{subsec:transvection-good-set}
Let \(V^*\) be the dual of the vector space $V$.  For \(\xi\in V^*\), define
\[
    \chi_\xi(u)=(-1)^{\xi(u)},
    \qquad
    S_\xi(z)=\sum_{i=1}^n\chi_\xi(z_i).
\]
Define the good set
\begin{equation}\label{eq:tr-good-set}
    \Good_{\mathrm{tr}}
    =\left\{z\in\Stief{n}{k}:
    \max_{\xi\in V^*\setminus\{0\}}|S_\xi(z)|\le\frac n4\right\}.
\end{equation}

\begin{lemma}[Stationary size of the transvection good set]\label{lem:tr-good-stationary}
There are absolute constants \(\varepsilon_0>0\) and \(c_0>0\) such that, if
\(k\le\varepsilon_0 n\), then
\[
    \pi(\Good_{\mathrm{tr}}^c)\le e^{-c_0 n} 
\]
and
\[
    \mu(\Good_{\mathrm{tr}}^{c}) \le e^{-c_{0} n}.
\]
\end{lemma}

\begin{proof}
Fix \(\xi\ne0\).  We claim that under \(\pi\), the projection
\[
    z\mapsto(\xi(z_1),\ldots,\xi(z_n))
\]
is uniform on \(\Ft^n\setminus\{0\}\).  To see this, choose coordinates so that
\(\xi\) is projection onto the first coordinate.  For any two nonzero vectors
\(q,q'\in\Ft^n\), choose \(A\in\mathrm{GL}_n(\Ft)\) with \(Aq=q'\).  Left
multiplication by \(A\) bijects the full-rank \(n\times k\) matrices whose first
column is \(q\) with those whose first column is \(q'\).  Hence the number of
rank-\(k\) completions is independent of \(q\), and the pushforward of \(\pi\)
is uniform on \(\Stief{n}{1}=\Ft^n\setminus\{0\}\).

Consequently, \(S_\xi\) has the law of a sum of \(n\) independent Rademacher
variables, conditioned on the vector not being identically zero.  The
conditioning event has probability at least \(1/2\), so Hoeffding's inequality
gives
\[
    \pi(|S_\xi|>n/4)\le Ce^{-cn}.
\]
A union bound over the \(2^k-1\) choices of $\xi \in V^{\ast} \backslash \{0\}$ gives
\[
    \pi(\Good_{\mathrm{tr}}^c)\le C2^k e^{-cn}.
\]
Choose \( 0 < \varepsilon_0 < 1\) small enough that
\(\varepsilon_0\log2<c/2\), and then choose \(c_0<c/2\).  If
\(k\le\varepsilon_0n\), the last display is at most \(e^{-c_0n}\) for all large
\(n\), after decreasing \(c_0\) if necessary to absorb the leading constant.

The second bound is very similar. Let \(Z\sim\mu\) and observe that
\[
    \left\{z\in V^n:
    \max_{\xi\in V^*\setminus\{0\}} |S_\xi(z)|\le \frac n4
    \right\}
    \subseteq \Stief{n}{k}.
\]
Indeed, if \(z\notin\Stief{n}{k}\), then
\(\Span(z_1,\ldots,z_n)\ne V\), so there is some nonzero
\(\xi\in V^*\) such that \(\xi(z_i)=0\) for every \(i\).  For this \(\xi\),
\[
    S_\xi(z)=\sum_{i=1}^n (-1)^{\xi(z_i)}=n,
\]
which is larger than \(n/4\).  Hence, as subsets of the ambient space \(V^n\),
\[
    \Good_{\mathrm{tr}}
    =
    \left\{z\in V^n:
    \max_{\xi\in V^*\setminus\{0\}} |S_\xi(z)|\le \frac n4
    \right\}.
\]

Now fix \(0\ne\xi\in V^*\).  Under \(\mu\), the random variables
\(\chi_\xi(Z_1),\ldots,\chi_\xi(Z_n)\) are independent Rademacher variables.
Thus Hoeffding's inequality gives
\[
    \mu\left(|S_\xi(Z)|>\frac n4\right)\le C e^{-cn}.
\]
Taking a union bound over the \(2^k-1\) nonzero functionals gives
\[
    \mu(\Good_{\mathrm{tr}}^c)
    \le
    \sum_{\xi\in V^*\setminus\{0\}}
    \mu\left(|S_\xi(Z)|>\frac n4\right)
    \le C2^k e^{-cn}.
\]
After decreasing \(\varepsilon_0\) and \(c_0\), if necessary, this is at most
\(e^{-c_0n}\) whenever \(k\le\varepsilon_0 n\), for all sufficiently large
\(n\).
\end{proof}

\subsection{Burn-in}\label{subsec:transvection-burnin}

We will prove our burn-in result by finding many copies of the $1$-column walk inside of the $k$-column walk. We repeatedly use the following mixing estimate from \cite[Theorem~1.1]{PeresTanakaZhai2020} (see also \cite[Theorem~1]{BenHamouPeres2018}):

\begin{lemma}\label{lem:tr-one-column-input}
For every fixed \(B<\infty\) there is a constant \(c_B<\infty\) such that the
one-column transvection walk on \(\Ft^n\setminus\{0\}\), started from any nonzero
vector, is within \(n^{-B}\) of stationarity after \(c_Bn\log n\) steps.
\end{lemma}

Our main burn-in estimate is: 

\begin{lemma}[Burn-in and occupancy of the good set]\label{lem:tr-burnin}
Let \(\varepsilon_0\) be as in Lemma~\ref{lem:tr-good-stationary}.  For every
fixed \(M<\infty\) there is a constant \(C_M<\infty\) such that the following
holds.  Assume \(1\le k\le\varepsilon_0 n\) and set
\begin{equation}\label{eq:tr-tstar}
    t_*=C_M nk\log n.
\end{equation}
Then, uniformly in \(x\in\Stief{n}{k}\) and all integer times
\(t\ge t_*\),
\begin{equation}\label{eq:tr-burnin-one-time}
    \Prob_x(X_t\notin\Good_{\mathrm{tr}})\le n^{-M}.
\end{equation}
Consequently, if \(L\le n^3\) and \(M=10\), then uniformly in \(x\) and
\(s\ge t_*\),
\begin{equation}\label{eq:tr-sojourn}
    \Prob_x(\exists\,u\in\{0,1,\ldots,L\}:X_{s+u}\notin\Good_{\mathrm{tr}})=o(1).
\end{equation}
\end{lemma}

\begin{proof}
For each nonzero \(\xi\in V^*\), the projected process
\[
    \Pi_\xi(X_t)=(\xi(X_t[1]),\ldots,\xi(X_t[n]))
\]
is a nonzero vector, since \(X_t\in\Stief{n}{k}\) and the rows span \(V\).  If
row \(a\) is added to row \(b\), then
\[
    \xi(X_{t+1}[b])=\xi(X_t[b])+\xi(X_t[a]),
\]
so \(\Pi_\xi(X_t)\) is exactly the one-column transvection walk on
\(\Ft^n\setminus\{0\}\).

Let \(P_1\) be the one-column kernel and let \(\pi_1\) be its stationary distribution.  Set
\[
    d(t)=\sup_x\|P_1^t(x,\cdot)-\pi_1\|_{\TV},
    \qquad
    \bar d(t)=\sup_{x,y}\|P_1^t(x,\cdot)-P_1^t(y,\cdot)\|_{\TV}.
\] 
By Lemma~\ref{lem:tr-one-column-input}, for every fixed \(B<\infty\) there is
\(c_B<\infty\) such that \(d(c_Bn\log n)\le n^{-B}\).  Since
\(\bar d(s+t)\le\bar d(s)\bar d(t)\) and \(\bar d(t)\le2d(t)\), after
\(\ell c_Bn\log n\) steps the distance is at most \((2n^{-B})^\ell\).  Taking
\(B=M+2\) and \(\ell=k\), and then choosing \(C_M\) large enough in
\eqref{eq:tr-tstar}, ensures that, for every fixed \(\xi\ne0\), the law of
\(\Pi_\xi(X_t)\) is within \(2^{-k}n^{-M-1}\) of stationarity for all integer
times \(t\ge t_*\) and all sufficiently large \(n\).
Therefore, using the stationary bound from the proof of Lemma~\ref{lem:tr-good-stationary},
\[
    \Prob_x(|S_\xi(X_t)|>n/4)
    \le Ce^{-cn}+2^{-k}n^{-M-1}.
\]
Taking a union bound over \(\xi\ne0\) gives
\[
    \Prob_x(X_t\notin\Good_{\mathrm{tr}})
    \le C2^k e^{-cn}+n^{-M-1}
    \le n^{-M}
\]
for all large \(n\).  This proves \eqref{eq:tr-burnin-one-time}.  Finally,
\eqref{eq:tr-sojourn} follows by applying \eqref{eq:tr-burnin-one-time} at all
times \(s,s+1,\ldots,s+L\) and using \(L\le n^3\) with \(M=10\).
\end{proof}

\subsection{Spectral gap of the one-row convolution kernel}\label{subsec:transvection-fibre-gap}

\begin{proposition}[Transvection fibre spectral gap]\label{prop:tr-fibre-gap}
For all sufficiently large \(n\), the following holds.  Fix a recipient
coordinate \(i\) and the non-recipient rows \((z_j)_{j\ne i}\).  If the corresponding
\(i\)-fibre intersects \(\Good_{\mathrm{tr}}\), then the fibre kernel
\eqref{eq:tr-fibre-kernel} on \(V=\Ft^k\) has spectral gap at least \(1/2\).
\end{proposition}

\begin{proof}
Choose a point in the fibre,
\[
    z=(z_1,\ldots,z_{i-1},u_0,z_{i+1},\ldots,z_n)\in\Good_{\mathrm{tr}}.
\]
The kernel \(K_i\) is reversible with respect to uniform measure on \(V\), since
\(V\) has exponent two.  Its eigenfunctions are the characters \(\chi_\xi\),
with eigenvalues
\[
    \widehat\nu_i(\xi)=\frac1{n-1}\sum_{j\ne i}\chi_\xi(z_j).
\]
For \(\xi\ne0\), since \(z\in\Good_{\mathrm{tr}}\),
\[
    \left|\sum_{j\ne i}\chi_\xi(z_j)\right|
    =|S_\xi(z)-\chi_\xi(u_0)|
    \le \frac n4+1.
\]
Thus, for all sufficiently large \(n\),
\[
    |\widehat\nu_i(\xi)|\le\frac{n/4+1}{n-1}\le\frac12,
    \qquad \xi\ne0.
\]
Every nonconstant eigenvalue is therefore at most \(1/2\) in absolute value,
so the spectral gap is at least \(1/2\).
\end{proof}

\subsection{Application of \texorpdfstring{Theorem~\ref*{thm:abstract-good-set}}{Theorem 2.7}}\label{subsec:transvection-apply}

We prove our main result for the transvection walk:

\begin{theorem}\label{thm:tr-linear-lazy}
There exist absolute constants \(\varepsilon_0>0\) and \(C_{\mathrm{lin}}<\infty\)
such that, for all sufficiently large \(n\) and all \(1\le k\le\varepsilon_0n\),
the half-lazy kernel
\[
    Q_{n,k}=\frac12(I+P_{n,k})
\]
satisfies
\[
    \tau_{\mathrm{mix}}(Q_{n,k})\le C_{\mathrm{lin}}nk\log n.
\]
\end{theorem}

\begin{proof}
We check that we can satisfy the assumptions of Theorem~\ref{thm:abstract-good-set}, with the choices
\[
    S=\Ft^k,
    \qquad N=n,
    \qquad \Omega=\Stief{n}{k},
    \qquad G=\Good_{\mathrm{tr}}.
\]
By Propositions~\ref{prop:tr-fibre-gap} and \ref{prop:abstract-fibre-gap-lsi}, and using the fact that \(\log|S|=k\log2\), there exists a universal constant $0 < C< \infty$ such that every nonnegative
\(F:V^n\to\mathbb R\) supported on \(G\) satisfies
\begin{equation} \label{IneqLocalLogSob}
    \Ent_\mu(F^2)\le C nk\,\cE_P^\mu(F,F).
\end{equation}
Thus we may take \(A=Cnk\) in the statement of Theorem~\ref{thm:abstract-good-set}.

Lemma~\ref{lem:tr-good-stationary} gives
\[
    \pi(G^c)=o(1),
    \qquad
    A\frac{\pi(G^c)}{\pi(G)}
    \le Cnk e^{-c_0n}=o(1).
\]
Take \(L=n^3\).  With \(M=10\), Lemma~\ref{lem:tr-burnin} gives a burn-in
\(t_*=O(nk\log n)\) and the sojourn estimate \eqref{eq:tr-sojourn}; in the
notation of Theorem~\ref{thm:abstract-good-set}, this means \(\eta=o(1)\).
The confinement time in \eqref{eq:abstract-tconf} satisfies
\[
    t_{\mathrm{conf}}
    \le Cnk\log(e+\log|\Stief{n}{k}|)
    \le Cnk\log n,
\]
since \(k\le n\).  In particular \(t_{\mathrm{conf}}=O(n^2\log n)=o(n^3)\), so
\(\zeta=\Prob(\operatorname{Poi}(t_{\mathrm{conf}})>n^3)=o(1)\).  Therefore the
right-hand side of \eqref{eq:abstract-after-burnin} is \(o(1)\).  Since
\(t_*\to\infty\), the Poisson lower tail in
\eqref{eq:abstract-continuous-condition} is also \(o(1)\).  Theorem~\ref{thm:abstract-good-set}
therefore gives mixing of the continuous-time chain in time \(O(nk\log n)\), and the
standard continuous-time/lazy-chain comparison gives the same order for the half-lazy chain.
\end{proof}

\begin{proof}[Proof of Theorem~\ref{thm:intro-transvection-main}]
Let \(\varepsilon_0\) be the constant from Theorem~\ref{thm:tr-linear-lazy}.  If
\(k\le\varepsilon_0 n\), the result is exactly Theorem~\ref{thm:tr-linear-lazy}.

If \(k>\varepsilon_0 n\), we note that the $k$-column walk is a projection of the full $n$-column walk. Since projecting cannot increase mixing times, the mixing time of the $k$-column walk is bounded from above by the \(O(n^2\log n)\) upper bound from \cite[Theorem~1]{BenHamou2025}.  Passing to the half-lazy version changes the bound by at most a universal constant. Thus
\[
    \tau_{\mathrm{mix}}(Q_{n,k})\le C_{\mathrm{full}}n^2\log n
    \le \frac{C_{\mathrm{full}}}{\varepsilon_0}kn\log n.
\]
Enlarging the constant proves the theorem.
\end{proof}

For completeness, we give the lower bounds as well, though both are essentially well-known from \cite{BenHamou2025,BenHamouPeres2018,PeresTanakaZhai2020}:

\begin{proof}[Proof of Theorem~\ref{thm:intro-transvection-lower}]
We prove the two lower bounds separately.

To prove the first, consider the map $f$ from $\Stief{n}{k}$ to $\Stief{n}{1}$ given by simply taking the first column. If $\{Z_{t}\}$ is a copy of the $k$-column walk, then $\{f(Z_{t})\}$ is exactly a copy of the $1$-column walk. This immediately implies that the mixing time of $\{Z_{t}\}$ is greater than or equal to the mixing time of $\{f(Z_{t})\}$. The latter is $\Omega(n \log(n))$ by \cite{BenHamouPeres2018}, proving the first lower bound.

To prove the second, we use a counting bound.  The state space has
\[
    |\Stief{n}{k}|=\prod_{q=0}^{k-1}(2^n-2^q)
    =2^{nk}\prod_{q=0}^{k-1}(1-2^{q-n}).
\]
Since \(k\le n\),
\[
    \prod_{q=0}^{k-1}(1-2^{q-n})
    \ge \prod_{j=1}^{\infty}(1-2^{-j})=:c_0>0,
\]
and hence
\begin{equation}\label{eq:tr-state-space-lower}
    |\Stief{n}{k}|\ge c_0 2^{nk}.
\end{equation}
From any fixed state, one step of the half-lazy chain has at most
\[
    M_n=1+n(n-1)
\]
possible outcomes.  Therefore, after \(t\) steps from a fixed starting point,
the law is supported on a set \(A_t\subseteq\Stief{n}{k}\) with
\(|A_t|\le M_n^t\).  Since \(\pi\) is uniform,
\[
    \pi(A_t)\le \frac{M_n^t}{|\Stief{n}{k}|}.
\]
If
\[
    t\le c'\frac{nk}{\log(e n)}
\]
with \(c'>0\) sufficiently small, then \eqref{eq:tr-state-space-lower} and
\(\log M_n=O(\log(e n))\) imply \(\pi(A_t)\le1/2\) for all sufficiently large
\(n\).  Since the chain is supported on \(A_t\) at time \(t\),
\[
    \|\delta_xQ_{n,k}^t-\pi\|_{\mathrm{TV}}
    \ge 1-\pi(A_t)\ge \frac12.
\]
This proves
\(\tau_{\mathrm{mix}}\ge c'nk/\log(e n)\).  Combining the two bounds and
decreasing the constant proves the theorem.
\end{proof}

\section{Power-averaged product replacement algorithm on Heisenberg groups}\label{sec:heis}

We describe a small modification of the usual product-replacement algorithm (PRA; see \cite{CellerLeedhamGreenMurrayNiemeyerOBrien1995,Pak2001}) on the Heisenberg group, which we call the \textit{power-averaged} PRA (PA-PRA).

\begin{remark} [Why Change the PRA?]
The PA-PRA algorithm defined in Equation \eqref{eq:heis-kernel-P} is not more challenging to implement or run than the usual PRA algorithm. We also expect it to be slightly more efficient; it is similar in spirit to lifted or momentum-based Markov chains such as \cite{DiaconisHolmesNeal2000}. In these senses, the PA-PRA is a natural generalization of the usual PRA.

However, we introduced this walk primarily for technical reasons. Taking random powers substantially simplifies several steps in the proof, for essentially the same reason that Fourier analysis of random walks on groups is much easier for conjugacy-invariant walks. See \textit{e.g.} \cite{Diaconis1988,DiaconisShahshahani1981} for explanations of technical problems in applying Fourier analysis to random walks on groups without these sorts of symmetries.
\end{remark}

\subsection{Notation for Heisenberg Group}

We give basic notation and facts about the Heisenberg group. None of the group-theoretic results in this paper are new, and we closely follow the notation and development of \cite{Misaghian2010}.

Throughout Section \ref{sec:heis}, fix an odd prime \(p\).  Let \(W\) be a vector space
of dimension \(h=2m\) over \(\Fp\), equipped with a nondegenerate alternating
form
\[
    \omega:W\times W\to\Fp.
\]

We denote the set associated with the Heisenberg group by
\[
    \Heis_{m}=\Heis(W)=W\times\Fp.
\]
We often omit the subscript $m$ when the discussion applies to all $m$. For an element $(v,z)\in\Heis$, we call $v$ the \textit{horizontal} coordinate and $z$ the \textit{central} coordinate. For an element \(g=(g_1,\ldots,g_r)\in\Heis^r\), we often write
\(g_i=(v_i,z_i)\), with \(v_i\in W\) and \(z_i\in\mathbb F_p\).

Multiplication in $\Heis$ is given by
\begin{equation}\label{eq:heis-law}
    (v,z)(w,t)=\left(v+w,\,z+t+\frac12\omega(v,w)\right).
\end{equation}
The identity is \((0,0)\), the inverse of \((v,z)\) is \((-v,-z)\), and the commutator is
\[
    [(v,z),(w,t)]=(0,\omega(v,w)).
\]

We collect some basic facts about $\Heis$ and related objects. Since \(\omega(v,v)=0\),
\begin{equation}\label{eq:heis-powers}
    (v,z)^a=(av,az),\qquad a\in\Fp.
\end{equation}
Thus \(\Heis\) has exponent \(p\). We claim that the center of \(H\) is
\[
    Z(H)=\{0\}\times \Fp.
\]
To see this, not that if \(v\ne0\), then by nondegeneracy of \(\omega\) there is
\(w\in W\) with \(\omega(v,w)\ne0\), so \((v,z)\) does not commute with
\((w,0)\). The quotient map
\[
    H\to H/Z(H)\simeq W,\qquad (v,z)\mapsto v,
\]
keeps only the horizontal coordinate. Also, since
\[
    [(v,z),(w,t)]=(0,\omega(v,w))
\]
and \(\omega\) is nondegenerate, the commutators generate all of
\(\{0\}\times\Fp\). Thus the abelianization is naturally identified with
\(W\).

For \(g_1,\ldots,g_r\in\Heis\), denote by \(\langle g_1,\ldots,g_r\rangle\) the subgroup of \(\Heis\) that they generate. Deliberately mimicking the notation for the transvection walk, set
\[
    \Omega=V_r(\Heis)=\{g=(g_1,\ldots,g_r)\in\Heis^r:
    \langle g_1,\ldots,g_r\rangle=\Heis\}.
\] 
We have:
\begin{lemma}\label{lem:heis-horizontal-generation}
 A tuple \(((v_1,z_1),\ldots,(v_r,z_r))\) generates \(\Heis\)
if and only if
\begin{equation}\label{eq:heis-generation-horizontal}
    \Span(v_1,\ldots,v_r)=W.
\end{equation}
\end{lemma}
\begin{proof}
The condition \eqref{eq:heis-generation-horizontal} is necessary because the
quotient map \(\Heis\to \Heis/(\{0\}\times\Fp)\cong W\) is given by \((v,z)\mapsto v\).

Conversely, assume that the horizontal vectors \(v_i\) span \(W\).
The subgroup \(H_0=\langle g_1,\ldots,g_r\rangle\), where
\(g_i=(v_i,z_i)\), maps onto \(W\) under the quotient map
\(\Heis\to W\). Since \(m\ge1\) and \(\omega\) is nondegenerate, there
exist \(v,w\in W\) with \(\omega(v,w)=1\).  By surjectivity, choose
\(x,y\in H_0\) with horizontal coordinates \(v\) and \(w\), respectively.
Writing \(x=(v,z)\) and \(y=(w,t)\), the commutator formula gives
\[
    [x,y]=(0,\omega(v,w))=(0,1).
\]
Thus \(H_0\) contains \(\{0\}\times\mathbb F_p\).  Since \(H_0\) also maps
onto \(W\), it follows that \(H_0=\Heis\).
\end{proof}

A \emph{symplectic basis} associated with the alternating form $\omega$ is a basis \(b_1,\ldots,b_h\) such that
\(\omega(b_{2q-1},b_{2q})=1\),
\(\omega(b_{2q},b_{2q-1})=-1\), and all other pairings vanish.  Such a basis
exists for every nondegenerate alternating form.  


\subsection{Notation for PA-PRA on the Heisenberg group}\label{subsec:heis-chain}

Let \(\pi\) be uniform on \(\Omega\), and let \(\mu\) be uniform on
\(\Heis^r\).  By \eqref{eq:heis-generation-horizontal}, conditioning \(\mu\) on
\(\Omega\) gives \(\pi\).

For distinct \(i,j\in[r]\) and \(a\in\Fp\), define the right and left maps $T^R_{j\to i,a},T^L_{j\to i,a}$ by the formulas:
\[
    T^R_{j\to i,a}g
    =(g_1,\ldots,g_{i-1},g_i g_j^a,g_{i+1},\ldots,g_r),
\]
and
\[
    T^L_{j\to i,a}g
    =(g_1,\ldots,g_{i-1},g_j^a g_i,g_{i+1},\ldots,g_r).
\]
In both cases, we call the index \(j\) the \textit{donor} and the index \(i\) the \textit{recipient}.

We claim that, for any $i,j,a$, the subgroup generated by the collections $T^R_{j\to i,a}g$ and $T^L_{j\to i,a}g$ is the same as the subgroup generated by the original collection $g$. To see this, note that the changed element \(g_i\) is a simple function of  the new
coordinate $g_j^a g_i$ or $g_{i} g_j^a$ and the unchanged donor $g_{j}$, and hence can be recovered when generating the subgroup. In particular, the maps both preserve \(\Omega\).

We define the non-lazy kernel for PA-PRA by:
\begin{equation}\label{eq:heis-kernel-P}
    Pf(g)=\frac1{2pr(r-1)}
    \sum_{i\ne j}\sum_{a\in\Fp}
    \left[f(T^R_{j\to i,a}g)+f(T^L_{j\to i,a}g)\right],
\end{equation}
and $Q = \frac{1}{2}(I + P)$ the usual $\frac{1}{2}$-lazy version.

Equation \eqref{eq:heis-kernel-P} also defines a kernel on all of \(\Heis^r\), not just $\Omega$.  Each move is a
bijection and its inverse is the corresponding move with exponent \(-a\) and
the same side.  Hence \(P\) is reversible with respect to \(\mu\) on
\(\Heis^r\) and with respect to \(\pi\) on \(\Omega\). 

We note a close connection between this walk and the transvection walk. If \(g_i=(v_i,z_i)\) and \(g_j=(v_j,z_j)\), then both left and right moves send
\begin{equation}\label{eq:heis-horizontal-update}
    v_i\leftarrow v_i+a v_j.
\end{equation}
Thus the horizontal projection is exactly the \(p\)-ary transvection walk on spanning
\(r\)-tuples in \(W\).

The main result of this section is the following upper bound. In reading this, we think of $r$ as roughly corresponding to $n$ in the transvection walk and $h$ as roughly corresponding to $k$; the requirement $h \leq \varepsilon r$ comes from the same calculation as the requirement that $k \leq \varepsilon n$ in the transvection walk.

\begin{theorem}\label{thm:heis-main}
Fix an odd prime \(p\) and \(\varepsilon\in(0,1)\).  There is a constant
\(C_{p,\varepsilon}<\infty\) such that the following holds.  Let
\(h=2m\ge2\) and suppose \(h\le\varepsilon r\).  For all sufficiently large
\(r\), the half-lazy PA-PRA walk on
\(V_r(\Heis)\) satisfies
\[
    \tau_{\mathrm{mix}}\le C_{p,\varepsilon} r(h+\log r)\log r.
\]
In particular, if \(h\ge\log r\), then
\[
    \tau_{\mathrm{mix}}\le C_{p,\varepsilon} rh\log r.
\]
\end{theorem}

At the end of the section, we give two lower bounds that are close to matching the upper bound when combined. 

\begin{theorem}[Heisenberg lower bounds]\label{thm:heis-lower}
Fix an odd prime \(p\).  There are constants \(c_p>0\) and \(r_0(p)<\infty\)
such that, for every \(h=2m\ge2\) and every \(r\ge\max\{h+1,r_0(p)\}\),
the half-lazy PA-PRA walk on \(V_r(\Heis)\) satisfies
\[
    \tau_{\mathrm{mix}}
    \ge c_p\max\left\{r\log r,\,\frac{r(h+1)}{\log(e r)}\right\}.
\]
\end{theorem}

\subsection{Basic Facts about the PA-PRA kernel}

We collect some basic facts.

\begin{lemma}[Connectivity]\label{lem:heis-connectivity}
Assume \(m\ge1\) and \(r\ge h+1\).  The PA-PRA
chain is irreducible on \(\Omega\).
\end{lemma}

\begin{proof}
Since the walk is reversible, it is enough to find a path from every \(g\in\Omega\) to one fixed element $g^\circ$.  Fix a symplectic basis \(b_1,\ldots,b_h\) of
\(W\), with \(\omega(b_1,b_2)=1\), and set the destination point
\[
    g^\circ=((b_1,0),\ldots,(b_h,0),(0,1),(0,0),\ldots,(0,0)).
\]
Now fix a generic point \(g=(g_1,\ldots,g_r)\in\Omega\), write \(g_i=(v_i,z_i)\), and set
\[
    B=(b_1,\ldots,b_h,0,\ldots,0)\in W^r.
\]
We first write a sequence of moves that start at $g$ and end at a point whose horizontal tuple is \(B\).  Since
\((v_1,\ldots,v_r)\) has rank \(h\), it is possible to choose \(M_0\in\mathrm{GL}_r(\Fp)\) with
\(M_0(v_1,\ldots,v_r)=B\).  Put \(d=\det M_0\).  Since \(r>h\), the diagonal
matrix
\[
    D=\operatorname{diag}(1,\ldots,1,d^{-1},1,
    \ldots,1),
\]
with \(d^{-1}\) in position \(h+1\), fixes \(B\).  Thus
\(M=DM_0\in\mathrm{SL}_r(\Fp)\) also satisfies \(M(v_1,\ldots,v_r)=B\).  Since
\(\mathrm{SL}_r(\Fp)\) is generated by elementary transvections, there exists a sequence of right moves (that is, moves of the form $T^R_{j\to i,a}$) that act on the horizontal parts of $g$ just like the matrix $M$. In particular, they transform the
horizontal tuple to \(B\).

After reducing to the case that the vector of horizontal elements is in $B$, keep the notation \(g\) for the current state.
Its horizontal tuple is \(B\), so \(g_{h+1}=(0,z_{h+1})\).  We next change
this coordinate to \((0,1)\).  For any \(a\in\Fp\), four right moves on the
recipient coordinate \(h+1\), with successive donors \(1,2,1,2\) and exponents
\(a,1,-a,-1\), replace \(g_{h+1}\) by
\[
    g_{h+1}g_1^a g_2 g_1^{-a}g_2^{-1}
    =g_{h+1}[g_1^a,g_2]
    =g_{h+1}(0,a).
\]
Taking \(a=1-z_{h+1}\) makes \(g_{h+1}=(0,1)\), without changing any horizontal coordinate.

We can now use the element \(g_{h+1}=(0,1)\) as a donor to standardize all other rows. For \(1\le i\le h\), write
\(g_i=(b_i,z_i)\) and replace \(g_i\) by \(g_i g_{h+1}^{-z_i}\).  For
\(i>h+1\), write \(g_i=(0,z_i)\) and replace \(g_i\) by
\(g_i g_{h+1}^{-z_i}\).  These moves send \(g\) to \(g^\circ\).  Reversing the
sequence connects \(g^\circ\) to every state, so the chain is irreducible on
\(\Omega\).
\end{proof}

For a fixed recipient coordinate \(i\) and frozen donors \((g_j)_{j\ne i}\),
conditioning on \(i\) being the recipient and projecting to the fibre gives the one-coordinate fibre kernel
\begin{equation}\label{eq:heis-fibre-kernel}
    K_i\varphi(x)
    =\frac1{2(r-1)}\sum_{j\ne i}\frac1p\sum_{a\in\Fp}
    \left[\varphi(xg_j^a)+\varphi(g_j^a x)\right],
    \qquad x\in\Heis.
\end{equation}
Equivalently, define the operators $R_{\nu}, L_{\nu}$ by the formulas
\begin{equation} \label{EqLeftRightDef}
R_\nu f(x)=\sum_{g\in\Heis}\nu(g)f(xg),
\qquad
L_\nu f(x)=\sum_{g\in\Heis}\nu(g)f(gx).
\end{equation}
Then 
\[
    K_i=\frac12(R_{\nu_i}+L_{\nu_i}),
    \qquad
    \nu_i=\frac1{r-1}\sum_{j\ne i}u_{\langle g_j\rangle},
\]
where \(u_{\langle g_j\rangle}\) is the uniform measure on the cyclic subgroup
\(\langle g_j\rangle\).

\subsection{The horizontal good set and stationary size}\label{subsec:heis-good-set}

For \(g=(g_1,\ldots,g_r)\in\Heis^r\), write \(g_i=(v_i,z_i)\).  For
\(\xi\in W^*\), define
\[
    N_\xi(g)=|\{i\in[r]:\xi(v_i)=0\}|
\]
and the good set 
\begin{equation}\label{eq:heis-good-set}
    \Good_{\mathrm H}
    =\left\{g\in\Omega:
    N_\xi(g)\le\beta_0r
    \text{ for every }0\ne\xi\in W^*\right\}.
\end{equation}

We now fix various special constants that will be used as scales for the burn-in argument and the remainder of the proof. Fix \(\varepsilon\in(0,1)\).  Define the function
\[
    I_p(\beta)
    =\beta\log(\beta p)
     +(1-\beta)\log\left(\frac{(1-\beta)p}{p-1}\right),
    \qquad \frac1p<\beta<1,
\]
which is the large-deviation rate for the upper tail of a
\(\operatorname{Bin}(r,1/p)\) random variable.  Since
\(I_p(\beta)\uparrow\log p\) as \(\beta\uparrow1\), it is possible to choose
\(\beta_0=\beta_0(p,\varepsilon)\) so that
\begin{equation}\label{eq:heis-beta-choice}
    \frac1p<\beta_0<1,
    \qquad
    \varepsilon\log p<I_p(\beta_0).
\end{equation}
Set
\begin{equation}\label{eq:heis-beta0-beta1}
    \beta_1=\frac12(1+\beta_0)
\end{equation}
and
\[
    \alpha_0=1-\beta_0,
    \qquad
    \alpha_* = \frac{p-1}{p}.
\]
Define
\begin{equation}\label{eq:heis-J-rate}
    J_p(a,b)=\int_a^b \log\left(\frac{(p-1)(1-u)}{u}\right)\,du,
    \qquad 0<a<b<\alpha_*.
\end{equation}
A direct calculation gives
\begin{equation}\label{eq:heis-J-equals-I}
    J_p(\alpha_0,\alpha_*)=I_p(\beta_0).
\end{equation}
By \eqref{eq:heis-beta-choice}, it is possible to choose \(\alpha_1\in(\alpha_0,\alpha_*)\) and
\(\eta_0>0\) such that
\begin{equation}\label{eq:heis-alpha1-choice}
    J_p(\alpha_0,\alpha_1)\ge \varepsilon\log p+3\eta_0.
\end{equation}

\begin{lemma}[Stationary size of the good set]\label{lem:heis-good-stationary}
Fix $p, \varepsilon$ as above. There is a constant \(c_{p,\varepsilon}>0\) such that, for all \(h\le\varepsilon r\),
\[
    \pi(\Good_{\mathrm H}^c)\le e^{-c_{p,\varepsilon}r}.
\]
Similarly, if \(\mu\) is uniform on \(\Heis^r\), then
\[
    \mu(\Good_{\mathrm H}) \geq 1-o(1).
\]
\end{lemma}

\begin{proof}
Under \(\pi\), the horizontal tuple \((v_1,\ldots,v_r)\) is uniform on
\[
    \operatorname{St}_{\Heis}(r,h)=\{(v_1,\ldots,v_r)\in W^r:\Span(v_1,\ldots,v_r)=W\},
\]
and, conditionally on this tuple, the central coordinates are independent and
uniform in \(\Fp\).

Fix \(0\ne\xi\in W^*\).  We claim that, under the uniform law on \(\operatorname{St}_{\Heis}(r,h)\),
the projection
\[
    (v_1,\ldots,v_r)\mapsto(\xi(v_1),\ldots,\xi(v_r))
\]
is uniform on \(\Fp^r\setminus\{0\}\).  To see this, note that after choosing coordinates so that
\(\xi\) is the first coordinate functional, the number of ways to complete a
fixed nonzero first coordinate column to a full-rank \(r\times h\) matrix is the
same for every nonzero column (this follows from the fact that \(\mathrm{GL}_r(\Fp)\) acts transitively on \(\Fp^r\setminus\{0\}\)).

Consequently, if $g \sim \pi$, then  \(N_\xi = N_{\xi}(g)\) has the law of the number of zero coordinates in an iid
uniform vector \(a \in \Fp^r\), conditioned on the event $\mathcal{E} = \{a \not\equiv 0\}$ that the vector $a$ is not identically
zero.  Since $\mathcal{E}$  has probability \(1-p^{-r}\), Chernoff's inequality and \eqref{eq:heis-beta-choice} give, for every sufficiently small
\(\eta_{p,\varepsilon}>0\),
\[
    \pi(N_\xi>\beta_0 r)
    \le C_p\exp\{-(I_p(\beta_0)-\eta_{p,\varepsilon})r\}.
\]
Choose \(\eta_{p,\varepsilon} > 0\) sufficiently small that this bound applies and also 
\(\varepsilon\log p<I_p(\beta_0)-2\eta_{p,\varepsilon}\).  A union bound over the
\(p^h-1\) nonzero linear functionals $\xi \in W^{\ast}$ gives
\[
    \pi(\Good_{\mathrm H}^c)
    \le C_p\exp\{h\log p-(I_p(\beta_0)-\eta_{p,\varepsilon})r\}
    \le e^{-c_{p,\varepsilon}r}
\]
for some \(c_{p,\varepsilon}>0\), whenever \(h\le\varepsilon r\) and \(r\) is
large.

For \(X\sim\mu\), the same Chernoff bound and union bound give
\[
    \mu\left(\exists\,0\ne\xi\in W^*:N_\xi(X)>\beta_0 r\right)
    \le p^h\exp\{-(I_p(\beta_0)-\eta_{p,\varepsilon})r\}.
\]
Also,
\[
\begin{aligned}
    \mu(\Heis^r\setminus\Omega)
    &=\Prob(\Span(v_1,\ldots,v_r)\ne W)\\
    &\le \sum_{0\ne\xi\in W^*}\Prob(\xi(v_1)=\cdots=\xi(v_r)=0)\\
    &\le p^h p^{-r}.
\end{aligned}
\]
Therefore
\[
    \mu(\Good_{\mathrm H}^c)
    \le p^h\exp\{-(I_p(\beta_0)-\eta_{p,\varepsilon})r\}+p^{h-r}.
\]
Because \(h\le\varepsilon r\) with \(\varepsilon<1\), the last display is
\(o(1)\). Hence \(\mu(\Good_{\mathrm H})=1-o(1)\).
\end{proof}

\subsection{Burn-in}\label{subsec:heis-burnin}

We prove that the PA-PRA spends most of its time in the good set \eqref{eq:heis-good-set} after a suitable burn-in period. We will again take advantage of copies of the one-column transvection walk.

Let \(Y_t\) be the following one-column \(p\)-ary transvection walk on
\(\Fp^r\setminus\{0\}\): at each step choose distinct \(i,j\in[r]\) and
\(a\in\Fp\) uniformly, and replace \(Y_i\) by \(Y_i+aY_j\).  Let
\[
    S_t=|\{i:Y_t(i)\ne0\}|.
\]

Note that \((S_t)\) is a birth-death chain on \(\{1,\ldots,r\}\).  From state \(s\),
its birth and death probabilities are
\begin{equation}\label{eq:heis-support-birth-death}
    B_s=\frac{p-1}{p}\frac{s(r-s)}{r(r-1)},
    \qquad
    D_s=\frac1p\frac{s(s-1)}{r(r-1)}.
\end{equation}
A birth occurs when the donor is nonzero, the recipient is zero, and \(a\ne0\);
a death occurs when donor and recipient are both nonzero and the unique value
of \(a\) making the recipient zero is chosen.

\begin{lemma}[Support growth]\label{lem:heis-support-growth}
Fix \(\alpha\in(0,\alpha_*)\).  There is a constant \(C_{p,\alpha}<\infty\) such
that, for all sufficiently large \(r\), the hitting time
\[
    \tau_\alpha=\inf\{t:S_t\ge\alpha r\}
\]
satisfies
\begin{equation}\label{eq:heis-support-growth-mean}
    \sup_{1\le s\le r}\E_s\tau_\alpha\le C_{p,\alpha}r\log r.
\end{equation}
Consequently, for a suitable \(B_{p,\alpha}<\infty\),
\begin{equation}\label{eq:heis-support-hit-block}
    \sup_{1\le s\le r}\Prob_s(\tau_\alpha>B_{p,\alpha}r\log r)\le\frac12.
\end{equation}
\end{lemma}

\begin{proof}
Set \(A=\lceil\alpha r\rceil\).  If \(s\ge A\), then \(\tau_\alpha=0\), so
we consider \(1\le s<A\).  We first control the birth-death ratios below
\(A\).  For \(1\le s<A\), set
\begin{equation} \label{EqBDRatio}
    \rho_s=\frac{D_s}{B_s}=\frac{s-1}{(p-1)(r-s)}.
\end{equation}
Since \(\alpha<\alpha_*=(p-1)/p\), there is \(q=q(p,\alpha)<1\) such that
\(\rho_s\le q\) for all \(s<A\) and all large \(r\).  Also
\[
    B_s\ge c_{p,\alpha}\frac{s}{r},\qquad 1\le s<A.
\]
Let \(\tau_A=\inf\{t:S_t=A\}\).  Since the chain moves by at most one at
each step and \(A=\lceil \alpha r\rceil\), this agrees with \(\tau_\alpha\) on
the event \(S_0<A\).  For \(1\le s\le A\), write
\[
    e_s=\E_s\tau_A ,
    \qquad e_A=0,
\]
and put \(d_s=e_s-e_{s+1}\) for \(1\le s\le A-1\). We have the following recurrence for \(1\le s\le A-1\),
\begin{equation}\label{eq:heis-support-hitting-recursion}
    B_s(e_s-e_{s+1})+D_s(e_s-e_{s-1})=1,
\end{equation}
with the convention that the term involving \(e_0\) is absent when \(s=1\).
Equivalently, since \(D_1=0\),
\[
    d_s=\frac1{B_s}+\rho_s d_{s-1},
    \qquad
    \rho_s=\frac{D_s}{B_s},
    \qquad d_0=0.
\]
Iterating this recursion yields
\begin{equation}\label{eq:heis-support-hitting-formula}
    d_k
    =
    \sum_{\ell=1}^k
    \frac1{B_\ell}
    \prod_{m=\ell+1}^k \rho_m ,
    \qquad 1\le k\le A-1,
\end{equation}
where the empty product is interpreted as \(1\).  Since
\(e_s=\sum_{k=s}^{A-1}d_k\), \eqref{eq:heis-support-hitting-formula} gives
\[
\begin{aligned}
    e_s
    &\le
    \sum_{k=1}^{A-1}
    \sum_{\ell=1}^k
    \frac1{B_\ell}
    \prod_{m=\ell+1}^k \rho_m  \\
    &\le
    C_{p,\alpha} r
    \sum_{k=1}^{A-1}
    \sum_{\ell=1}^k
    \frac{q^{k-\ell}}{\ell}  \\
    &\le
    C_{p,\alpha} r
    \sum_{\ell=1}^{A-1}
    \frac1\ell
    \sum_{k=\ell}^{A-1} q^{k-\ell}  \\
    &\le
    C_{p,\alpha} r
    \sum_{\ell=1}^{A-1}\frac1\ell
    \le C_{p,\alpha}r\log r .
\end{aligned}
\]
Here we used \(\rho_m\le q<1\) and
\(B_\ell\ge c_{p,\alpha}\ell/r\) for \(\ell<A\).  This proves
\eqref{eq:heis-support-growth-mean}.  Choosing \(B_{p,\alpha}\) sufficiently
large, Markov's inequality gives
\[
    \sup_{1\le s\le r}
    \Prob_s(\tau_\alpha>B_{p,\alpha}r\log r)
    \le \frac12,
\]
which is \eqref{eq:heis-support-hit-block}.
\end{proof}

\begin{lemma}[Downward crossing]\label{lem:heis-downward-crossing}
Let \(0<\alpha_0<\alpha_1<\alpha_*\), and define \(J_p\) by
\eqref{eq:heis-J-rate}.  For every \(T\le r^3\),
\[
    \sup_{s\ge\alpha_1 r}
    \Prob_s\left(\inf_{0\le t\le T}S_t<\alpha_0 r\right)
    \le T\exp\{-(J_p(\alpha_0,\alpha_1)+o(1))r\},
\]
where the implicit constant in the \(o(1)\) term depends only on \(p,\alpha_0,\alpha_1\) and, when those are fixed, tends to zero as
\(r\to\infty\).
\end{lemma}

\begin{proof}
It is enough to consider integer \(T\); otherwise replace \(T\) by
\(\lfloor T\rfloor\).  Set
\[
    A_0=\lfloor \alpha_0 r\rfloor,
    \qquad
    A_1=\lceil \alpha_1 r\rceil .
\]
Let \(\widehat S_n\) be the embedded jump chain obtained from \(S_t\). If \(S_t\) reaches \(\{1,\ldots,A_0\}\) within \(T\) steps, then \(\widehat S_n\) does as well, so it is enough to bound
\[
    \sup_{s\ge A_1}
    \Prob_s\left(\inf_{0\le n\le T}\widehat S_n\le A_0\right).
\]
For \(1\le m\le r-1\), set
\[
    \rho_m=\frac{D_m}{B_m}
    =
    \frac{m-1}{(p-1)(r-m)}.
\]
Let \(\widehat\tau_a=\inf\{n\ge0:\widehat S_n=a\}\).  The standard
birth-death hitting-probability formula gives, for \(A_0\le s\le A_1\),
\[
    \Prob_s(\widehat\tau_{A_0}<\widehat\tau_{A_1})
    =
    \frac{
        \sum_{k=s}^{A_1-1}\prod_{m=A_0+1}^k\rho_m
    }{
        \sum_{k=A_0}^{A_1-1}\prod_{m=A_0+1}^k\rho_m
    },
\]
with the empty product interpreted as \(1\).  In particular,
\[
    \Prob_{A_1-1}(\widehat\tau_{A_0}<\widehat\tau_{A_1})
    \le
    \prod_{m=A_0+1}^{A_1-1}\rho_m .
\]
Every path from \([A_1,\infty)\) to \(\{1,\ldots,A_0\}\) has an attempt
beginning with a jump from \(A_1\) to \(A_1-1\).  There are at most \(T\)
such attempts among the first \(T\) steps of \(\widehat S\).  By the strong Markov property and a union bound,
\begin{equation} \label{IneqCrossingPenultimate}
    \Prob_s\left(\inf_{0\le n\le T}\widehat S_n\le A_0\right)
    \le
    T\prod_{m=A_0+1}^{A_1-1}\rho_m .
\end{equation}

Finally, recalling that $\alpha_{0}>0$ and $\alpha_{1} < 1$,
\[
\begin{aligned}
 \log\left(\prod_{m=A_0+1}^{A_1-1}\rho_m\right) 
    &=
    -\sum_{m=A_0+1}^{A_1-1}
    \log\left(\frac{(p-1)(r-m)}{m-1}\right)  \\
    &=
    -r\int_{\alpha_0}^{\alpha_1}
    \log\left(\frac{(p-1)(1-u)}{u}\right)\,du+o(r)  \\
    &=
    -(J_p(\alpha_0,\alpha_1)+o(1))r .
\end{aligned}
\]
Combining with Inequality \eqref{IneqCrossingPenultimate}  completes the proof.
\end{proof}

\begin{lemma}\label{lem:heis-scalar-burnin}
For every fixed \(M<\infty\) there is a constant
\(C_{p,\varepsilon,M}<\infty\) such that the following holds.  Let \(Y_t\) be
the one-column \(p\)-ary transvection walk on \(\Fp^r\setminus\{0\}\), and let
\(Z_t=|\{i:Y_t(i)=0\}|\).  If
\[
    t_*=C_{p,\varepsilon,M}r\log r\,(h+M\log r),
    \qquad h\le\varepsilon r,
\]
then, uniformly in the initial state and in all \(t\ge t_*\),
\begin{equation}\label{eq:heis-scalar-burnin-tail}
    \Prob(Z_t>\beta_0 r)
    \le p^{-h}r^{-M-1}+\exp\{-(J_p(\alpha_0,\alpha_1)+o(1))r\}.
\end{equation}
\end{lemma}

\begin{proof}
It is enough to prove the estimate at \(t=t_*\), because the bound is uniform
in the starting state and later times can be handled by conditioning on the
state at time \(t-t_*\).  Let \(S_t=r-Z_t\), and set
\[
    \tau_1=\inf\{t:S_t\ge\alpha_1r\}.
\]
By Lemma~\ref{lem:heis-support-growth}, with \(\alpha=\alpha_1\), there is a
block length \(B=B_{p,\alpha_1}r\log r\) such that the probability of not
hitting \(\alpha_1r\) during one block is at most \(1/2\), uniformly in the
starting state.  Iterating over
\[
    L=\left\lfloor\frac{t_*}{B}\right\rfloor
\]
blocks gives \(\Prob(\tau_1>t_*)\le2^{-L}\).  Choosing
\(C_{p,\varepsilon,M}\) large enough gives 
\begin{equation} \label{IneqPart1Downcrossing2}
\Prob(\tau_1>t_*) \le p^{-h}r^{-M-1}.
\end{equation}

We now consider the event \(\{\tau_1\le t_*\}\). On this event, we have \(\{Z_{t_*}>\beta_0r\} = \{S_{t_*}<\alpha_0r\}\).  Since \(h\le\varepsilon r\),
\[
    t_*=O_{p,\varepsilon,M}(r^2\log r)\le r^3
\]
for all large \(r\).  By the strong Markov property at \(\tau_1\) and
Lemma~\ref{lem:heis-downward-crossing},
\[
\begin{aligned}
    \Prob(S_{t_*}<\alpha_0r,\ \tau_1\le t_*)
    &\le
    \sup_{s\ge\alpha_1r}
    \Prob_s\left(\inf_{0\le u\le t_*}S_u<\alpha_0r\right) \\
    &\le
    t_*\exp\{-(J_p(\alpha_0,\alpha_1)+o(1))r\}.
\end{aligned}
\]
Since \(t_*\le r^3\), the factor \(t_*\) is absorbed into the \(o(1)r\) term
in the exponent.  This gives the second term in \eqref{eq:heis-scalar-burnin-tail}; combining with Inequality \eqref{IneqPart1Downcrossing2} completes the proof.
\end{proof}

We conclude with the burn-in estimate:

\begin{lemma}\label{lem:heis-burnin}
For every fixed \(M<\infty\) there is a constant
\(C_{p,\varepsilon,M}<\infty\) such that the following holds.  If
\(h\le\varepsilon r\) and
\begin{equation}\label{eq:heis-tstar}
    t_*=C_{p,\varepsilon,M}r\log r\,(h+M\log r),
\end{equation}
then, uniformly in \(x\in\Omega\) and all \(t\ge t_*\),
\begin{equation}\label{eq:heis-burnin-one-time}
    \Prob_x(X_t\notin\Good_{\mathrm H})\le r^{-M}.
\end{equation}
Consequently, if \(L\le r^3\) and \(M=10\), then uniformly in \(x\) and
\(s\ge t_*\),
\begin{equation}\label{eq:heis-sojourn}
    \Prob_x(\exists\,u\in\{0,1,\ldots,L\}:X_{s+u}\notin\Good_{\mathrm H})=o(1).
\end{equation}
\end{lemma}

\begin{proof}
For each nonzero \(\xi\in W^*\), the projected process
\[
    \Pi_\xi(X_t)=(\xi(v_1(t)),\ldots,\xi(v_r(t)))
\]
is exactly the one-column \(p\)-ary transvection walk on
\(\Fp^r\setminus\{0\}\), because the horizontal update is
\(v_i\leftarrow v_i+a v_j\).  Since \(X_t\in\Omega\), the vector
\(\Pi_\xi(X_t)\) is never identically zero.

By Lemma~\ref{lem:heis-scalar-burnin}, for every fixed \(\xi\ne0\),
\[
    \Prob_x(N_\xi(X_t)>\beta_0r)
    \le p^{-h}r^{-M-1}+\exp\{-(J_p(\alpha_0,\alpha_1)+o(1))r\},
    \qquad t\ge t_*.
\]
Taking a union bound over at most \(p^h\) nonzero functionals $\xi$ gives
\[
\begin{aligned}
    \Prob_x(X_t\notin\Good_{\mathrm H})
    &\le r^{-M-1}+p^h\exp\{-(J_p(\alpha_0,\alpha_1)+o(1))r\}\\
    &\le r^{-M-1}+\exp\{-2\eta_0r+o(r)\}
    \le r^{-M},
\end{aligned}
\]
where we used \(h\le\varepsilon r\) and \eqref{eq:heis-alpha1-choice}.  This
proves \eqref{eq:heis-burnin-one-time}.  The sojourn estimate in Inequality \eqref{eq:heis-sojourn} follows by
applying \eqref{eq:heis-burnin-one-time} at all times
\(s,s+1,\ldots,s+L\) and using \(L\le r^3\) with \(M=10\).
\end{proof}

\subsection{The Heisenberg fibre spectral estimate}\label{subsec:heis-fibre-spectrum}
In this section, we prove the analogue of Proposition~\ref{prop:tr-fibre-gap}. The main difference between the PA-PRA and the transvection walk is that we are no longer dealing with only one-dimensional representations. We begin by collecting the notation and necessary facts from the already-understood representation theory of the Heisenberg group, following \cite{Misaghian2010} very closely. None of the representation-theoretic results in this section are new.

Define the additive character
\[
    \psi(t)=e^{2\pi i t/p},\qquad t\in\mathbb F_p,
\]
where we identify \(\mathbb F_p\) with \(\{0,1,\ldots,p-1\}\) in the
exponent.

\begin{proposition}[Heisenberg representation facts]\label{prop:heis-reps}
Let \(\Heis=W\times\Fp\) be the group in \eqref{eq:heis-law}, with
\(\dim W=2m\).  For \(\xi\in W^*\), define
\[
    \chi_\xi(v,z)=\psi(\xi(v)).
\]
The following facts hold.
\begin{enumerate}[label=\textup{(\roman*)}, leftmargin=*]
    \item The functions \(\chi_\xi\), \(\xi\in W^*\), are exactly the
    irreducible complex representations of \(\Heis\) on which the center
    \(\{0\}\times\Fp\) acts trivially.
    \item For each \(\lambda\in\Fp^\times\), there is an irreducible unitary
    representation \(\rho_\lambda\) of \(\Heis\), of dimension \(p^m\), such
    that
    \[
        \rho_\lambda(0,z)=\psi(\lambda z)I,\qquad z\in\Fp.
    \]
    \item If \(\sigma\) is any irreducible complex representation of \(\Heis\)
    on which the center does not act trivially, then there is a unique
    \(\lambda\in\Fp^\times\) such that
    \[
        \sigma(0,z)=\psi(\lambda z)I,\qquad z\in\Fp.
    \]
    For this \(\lambda\), there is a unitary map \(U\) from the representation
    space of \(\sigma\) to that of \(\rho_\lambda\) such that
    \[
        U\sigma(g)U^{-1}=\rho_\lambda(g),\qquad g\in\Heis.
    \]
\end{enumerate}
Moreover, if \(g=(v,z)\), \(g'=(w,t)\), and \(\lambda\in\Fp^\times\), then
\begin{equation}\label{eq:heis-projective-commutation}
    \rho_\lambda(g)\rho_\lambda(g')
    =
    \psi(\lambda\omega(v,w))\rho_\lambda(g')\rho_\lambda(g).
\end{equation}
\end{proposition}

\begin{proof}
We use the representation theory of the finite Heisenberg group from
\cite{Misaghian2010}.  Misaghian works over a finite field $F$ of order
$q$, with $\dim_F W=2n$.  Here $F=\Fp$, $q=p$, and $n=m$.  His
multiplication rule is written using a symplectic form
$\langle\cdot,\cdot\rangle$; in our normalization this corresponds to
\[
    \langle v,w\rangle=\frac12\omega(v,w).
\]
Since $p$ is odd, this is again a nondegenerate alternating form, so his
results apply to the group law \eqref{eq:heis-law}.

We first discuss the representations on which the center acts trivially.  Any
irreducible representation with trivial action of $Z(\Heis)=\{0\}\times\Fp$
factors through
\[
    \Heis/Z(\Heis)\cong W.
\]
This quotient is abelian, so all of its irreducible complex representations are
one-dimensional characters.  The characters of the additive group $W$ are
exactly
\[
    v\mapsto \psi(\xi(v)),\qquad \xi\in W^*,
\]
and pulling these characters back along $(v,z)\mapsto v$ gives precisely the
functions $\chi_\xi$ in the statement.  This proves \textup{(i)}; this is also
recorded in \cite[Lemmas~5--6, Theorem~2, and Corollary~5,
pp.~164--166]{Misaghian2010}.

We now explain how \cite[Theorem~3]{Misaghian2010} gives the representations
in \textup{(ii)}.  Choose subspaces $A,B\le W$ such that
\[
    W=A\oplus B,\qquad \dim A=\dim B=m,
\]
and such that $\omega$ vanishes on $A\times A$ and on $B\times B$.  Such a
choice exists: after choosing a symplectic basis
$b_1,\ldots,b_{2m}$ as above, one may take
\[
    A=\Span(b_1,b_3,\ldots,b_{2m-1}),\qquad
    B=\Span(b_2,b_4,\ldots,b_{2m}).
\]
Set
\[
    K=A\times\Fp\le \Heis.
\]
The condition that $\omega$ vanishes on $A\times A$ implies that $K$ is a
subgroup of $\Heis$.  Fix $\lambda\in\Fp^\times$ and define a character of $K$ by
\[
    \eta_\lambda(a,z)=\psi(\lambda z),\qquad (a,z)\in K.
\]
Misaghian's space $C(H,K)$ is, in this notation,
\[
    E_\lambda
    =\{f:\Heis\to\C:
        f(kh)=\eta_\lambda(k)f(h)
        \text{ for all } k\in K,\ h\in \Heis\}.
\]
Thus the functions $f$ in Misaghian's formula are vectors in the representation
space.  For each fixed group element $g\in\Heis$, Misaghian defines the linear
operator
\[
    (\rho_\lambda(g)f)(h)=f(hg),\qquad f\in E_\lambda,\ h\in\Heis.
\]
Theorem~3 of \cite{Misaghian2010} says that this gives an irreducible
representation of degree $q^n=p^m$.

We make the dimension and the associated matrices explicit.  For each $b\in B$,
write
\[
    s_b=(b,0)\in\Heis.
\]
If $x=(w,z)\in\Heis$ and $w=a+b$ with $a\in A$ and $b\in B$, then
\[
    x=\left(a,z-\frac12\omega(a,b)\right)s_b.
\]
This decomposition is unique, so the elements $s_b$, $b\in B$, are a set of
representatives for the left cosets $K\backslash \Heis$.  Since $|B|=p^m$, there
are $p^m$ such cosets.  Define functions $e_b\in E_\lambda$, indexed by
$b\in B$, by
\[
    e_b(k s_{b'})=\eta_\lambda(k)\one_{\{b=b'\}},\qquad
    k\in K,\ b'\in B.
\]
The uniqueness of the decomposition shows that these functions are well-defined,
and they form a basis of $E_\lambda$: every $f\in E_\lambda$ is determined by the
$p^m$ values $f(s_b)$, and
\[
    f=\sum_{b\in B} f(s_b)e_b.
\]
Relative to this basis, the operator $\rho_\lambda(g)$ is a $p^m\times p^m$
matrix.  Its entries are obtained by applying the operator to the basis vectors:
\[
    M_{b,b'}(g)=(\rho_\lambda(g)e_{b'})(s_b)=e_{b'}(s_b g),\qquad b,b'\in B.
\]
Thus Misaghian's formula, although written as a formula for every vector $f$,
is exactly a formula for a linear operator, and a choice of the basis above turns
that operator into a matrix.

Now take a central element $c=(0,z)\in Z(\Heis)$.  Since $c\in K$ and $c$ is
central, for every $f\in E_\lambda$ and every $h\in\Heis$ we have
\[
    (\rho_\lambda(c)f)(h)
    =f(hc)
    =f(ch)
    =\eta_\lambda(c)f(h)
    =\psi(\lambda z)f(h).
\]
Therefore $\rho_\lambda(0,z)$ sends every vector of $E_\lambda$ to
$\psi(\lambda z)$ times itself.  In the basis $\{e_b:b\in B\}$, and hence in any
basis, this linear operator is
\[
    \rho_\lambda(0,z)=\psi(\lambda z)I.
\]
This is the central-character computation recorded in
\cite[Corollary~8, p.~169]{Misaghian2010}.  Since $\Heis$ is finite, any
finite-dimensional complex representation of $\Heis$ can be made unitary by
averaging an arbitrary inner product over the group.  Hence the representations
constructed above may be taken to be unitary, proving \textup{(ii)}.

It remains to explain the uniqueness statement in \textup{(iii)}.  If $\sigma$ is
irreducible, then Schur's lemma implies that every central element $(0,z)$ acts
by a scalar.  These scalars form an additive character of $\Fp$.  If the center
does not act trivially, this character is nontrivial, and the nontrivial additive
characters of $\Fp$ are exactly
\[
    z\mapsto \psi(\lambda z),\qquad \lambda\in\Fp^\times,
\]
with $\lambda$ uniquely determined.  Misaghian's Stone--von Neumann theorem,
\cite[Theorem~1, p.~162]{Misaghian2010}, says that for each such nontrivial
central character there is only one irreducible representation up to equivalence.
Thus $\sigma$ is equivalent to the representation $\rho_\lambda$ constructed
above.  Choosing invariant inner products on the two representation spaces, the
intertwining map can be chosen unitary.  Finally,
\cite[Corollary~10, p.~170]{Misaghian2010} says that the representations in
\textup{(i)} and \textup{(ii)} exhaust all irreducible representations of $\Heis$.

It remains only to check \eqref{eq:heis-projective-commutation}.  By
\eqref{eq:heis-law},
\[
    gg'=(v+w,z+t+\tfrac12\omega(v,w)),
    \qquad
    g'g=(v+w,z+t-\tfrac12\omega(v,w)).
\]
Hence
\[
    gg'=(0,\omega(v,w))g'g.
\]
Applying $\rho_\lambda$, and using
$\rho_\lambda(0,s)=\psi(\lambda s)I$, gives
\[
    \rho_\lambda(g)\rho_\lambda(g')
    =
    \psi(\lambda\omega(v,w))\rho_\lambda(g')\rho_\lambda(g),
\]
as claimed.
\end{proof}

Having described the representations, we now give a series of bounds that will help us show the power-averaging of PA-PRA helps. 

For a unitary matrix \(U\) with \(U^p=I\), define
\[
    P_U=\frac1p\sum_{a\in\Fp}U^a.
\]
This is the orthogonal projection onto the fixed space of \(U\) (see Lemma 2.2.2 of \cite{Sturmfels2008Algorithms}). The following lemma borrows calculations from \cite{Halmos1969TwoSubspaces}:

\begin{lemma}[Two projections]\label{lem:heis-two-projections}
Let \(U,V\) be unitary operators on a finite-dimensional complex Hilbert space,
with \(U^p=V^p=I\).  Suppose
\[
    UV=\zeta VU
\]
for some nontrivial \(p\)th root of unity \(\zeta\ne1\).  Then
\[
    \|P_UP_V\|_{\mathrm{op}}\le p^{-1/2}.
\]
\end{lemma}

\begin{proof}
For \(b\ne0\), the relation \(UV^b=\zeta^bV^bU\) implies that \(V^b\) maps the
\(U\)-fixed space into a nontrivial eigenspace of \(U\).  Therefore
\[
    P_UV^bP_U=0,\qquad b\ne0.
\]
It follows that
\[
    P_UP_VP_U=\frac1p\sum_{b\in\Fp}P_UV^bP_U=\frac1pP_U.
\]
Since \(P_U\) and \(P_V\) are orthogonal projections,
\((P_UP_V)^*=P_VP_U\).  Hence
\[
    \|P_UP_V\|_{\mathrm{op}}^2
    =\|(P_UP_V)(P_UP_V)^*\|_{\mathrm{op}}
    =\|P_UP_VP_U\|_{\mathrm{op}}
    =p^{-1}\|P_U\|_{\mathrm{op}}
    \le p^{-1}.
\]
Taking square roots proves the claim.
\end{proof}

\begin{lemma}[Average of projections]\label{lem:heis-average-projections}
Let \(P_1,\ldots,P_N\) be orthogonal projections on a finite-dimensional
complex Hilbert space.  Suppose that for at least 
 \(\delta {N \choose 2}\) of the ordered pairs \((j,\ell)\in[N]^2\) one has
\[
    \|P_jP_\ell\|_{\mathrm{op}}\le\alpha.
\]
Then
\[
    \left\|\frac1N\sum_{j=1}^NP_j\right\|_{\mathrm{op}}
    \le 1-\frac\delta2(1-\alpha).
\]
\end{lemma}

\begin{proof}
We first prove an estimate with only two projections.  Let \(P\) and \(Q\)
be orthogonal projections on a finite-dimensional complex Hilbert space, and
put \(s=\|PQ\|_{\operatorname{op}}\).  We claim that
\begin{equation}\label{eq:two-projection-average-bound}
    \lambda_{\max}\left(\frac{P+Q}{2}\right)
    \le \frac{1+s}{2}.
\end{equation}
To see this, let \(f\) be a unit eigenvector of \(P+Q\) with eigenvalue \(\lambda\).  If
\(\lambda\le1\), then \eqref{eq:two-projection-average-bound} is immediate.
Assume \(\lambda>1\).  Set \(u=Pf\) and \(v=Qf\).  Applying \(P\) and \(Q\) to
\((P+Q)f=\lambda f\) gives
\[
    Pv=(\lambda-1)u,
    \qquad
    Qu=(\lambda-1)v.
\]
Therefore
\[
    (\lambda-1)\|u\|=\|Pv\|\le s\|v\|,
    \qquad
    (\lambda-1)\|v\|=\|Qu\|\le s\|u\|,
\]
where the second inequality uses \(\|QP\|_{\operatorname{op}}=
\|(PQ)^*\|_{\operatorname{op}}=s\).  Since \(\lambda>1\), neither \(u\) nor
\(v\) is zero.  Multiplying the two inequalities gives \(\lambda-1\le s\),
and hence \eqref{eq:two-projection-average-bound}.

We now consider the general-$N$ case. Fix a unit vector \(f\).  Let \(J,L\) be independent uniform indices in
\([N]\), and let \(E\) be the event
\(\|P_JP_L\|_{\operatorname{op}}\le\alpha\).  Since \(\Prob(E)\ge\delta\),
\eqref{eq:two-projection-average-bound} and the trivial bound
\(\lambda_{\max}((P_J+P_L)/2)\le1\) give
\[
\begin{aligned}
    \left\langle\frac1N\sum_{j=1}^NP_j f,f\right\rangle
    &=\E\left[\frac{\langle P_Jf,f\rangle+\langle P_Lf,f\rangle}{2}\right]\\
    &\le \Prob(E^c)+\Prob(E)\frac{1+\alpha}{2}\\
    &\le 1-\frac\delta2(1-\alpha).
\end{aligned}
\]
Taking the supremum over unit vectors \(f\) proves the claim.

\end{proof}

We put these together to bound the spectral gaps of the fibre kernels:

\begin{proposition}\label{prop:heis-fibre-gap}
Fix \(\beta\in(1/p,1)\).  Let \(a_1,\ldots,a_N\in\Heis\), where
\(a_j=(v_j,z_j)\), and let \(\mu_W=N^{-1}\sum_j\delta_{v_j}\).  Suppose that \(\mu_W\) satisfies
\begin{equation}\label{eq:heis-hyperplane-balance-fibre}
    \mu_W(L)\le\beta
    \qquad\text{for every hyperplane }L\le W.
\end{equation}
Define
\[
\nu =
\frac1N\sum_{j=1}^N u_{\langle a_j\rangle}.
\]
Recall the definition \eqref{EqLeftRightDef}. Then the kernel $K=\frac12(R_\nu+L_\nu)$ has spectral gap at least
\begin{equation}\label{eq:heis-gamma-beta}
    \gamma_{p,\beta}
    =\min\left\{1-\beta,
    \frac{(1-\beta)^2}{2}(1-p^{-1/2})\right\}>0.
\end{equation}
\end{proposition}

\begin{proof}
We prove three estimates and then combine them. Define the functional $\widehat \nu$ by the formula:
\[
\widehat\nu(\rho)=\sum_{g\in\Heis}\nu(g)\rho(g).
\]
We use the notation for representations of $\Heis$ from the statement of Proposition \ref{prop:heis-reps} throughout the argument.

First, we control the one-dimensional representations.  Fix
\(0\ne\xi\in W^*\). Recall the following fact: if $x$ is a \(p\)th root of unity, then either $\sum_{j=1}^{p} x^{j} = p$ (if $x=1$) or   $\sum_{j=1}^{p} x^{j} = 0$ (otherwise). Using \eqref{eq:heis-powers} and then this fact, 
\[
\begin{aligned}
    \widehat\nu(\chi_\xi)
    &=\frac1N\sum_{j=1}^N\frac1p\sum_{s\in\Fp}\chi_\xi(a_j^s)\\
    &=\frac1N\sum_{j=1}^N\frac1p\sum_{s\in\Fp}\psi(s\xi(v_j))\\
    &=\mu_W(\ker \xi).
\end{aligned}
\]
Since \(\ker\xi\) is a hyperplane, the assumption in inequality \eqref{eq:heis-hyperplane-balance-fibre}
gives
\begin{equation} \label{IneqRepBound1}
    |\widehat\nu(\chi_\xi)|\le \beta .
\end{equation}

Second, we control the irreducible representations with nontrivial central
character.  Fix \(\lambda\in\Fp^\times\), and set
\[
    U_j=\rho_\lambda(a_j),\qquad
    P_j^{(\lambda)}=\frac1p\sum_{a\in\Fp}U_j^a,
    \qquad
    A_\lambda=\frac1N\sum_{j=1}^N P_j^{(\lambda)} .
\]
By \eqref{eq:heis-powers}, \(U_j^p=I\). By this identity and the fact that $U$ is unitary, we have that \(P_j^{(\lambda)}\) is the
orthogonal projection onto the space fixed by \(U_j\) (see Lemma 2.2.2 of \cite{Sturmfels2008Algorithms}). Furthermore,
\[
    \widehat\nu(\rho_\lambda)=A_\lambda.
\]

We now consider a pair $1 \leq j < \ell \leq N$. If \(\omega(v_j,v_\ell)\ne0\), then \eqref{eq:heis-projective-commutation}
gives
\[
    U_jU_\ell=\psi(\lambda\omega(v_j,v_\ell))U_\ell U_j,
\]
and the scalar \(\psi(\lambda\omega(v_j,v_\ell))\) is a nontrivial \(p\)th root
of unity. In this case, Lemma~\ref{lem:heis-two-projections} gives
\[
    \|P_j^{(\lambda)}P_\ell^{(\lambda)}\|_{\operatorname{op}}
    \le p^{-1/2}.
\]
We now show that this estimate applies to many ordered pairs.  Let
\(V,V' \sim \mu_W\) be independent, and define
\(q=\mu_W(\{0\})\).  Since \(\{0\}\) is contained in every hyperplane,
\(q\le\beta\). If \(v\ne0\), then the functional \(w\mapsto\omega(v,w)\) is nonzero by
nondegeneracy of \(\omega\).  Hence
\[
    v^\perp=\{w\in W:\omega(v,w)=0\}
\]
is a hyperplane of \(W\), and \eqref{eq:heis-hyperplane-balance-fibre}
gives \(\mu_W(v^\perp)\le\beta\).  Therefore
\[
\begin{aligned}
    \Prob(\omega(V,V')=0)
    &=q+\sum_{v\ne0}\mu_W(v)\mu_W(v^\perp)\\
    &\le q+(1-q)\beta
    \le \beta+(1-\beta)\beta
    =2\beta-\beta^2.
\end{aligned}
\]
Thus at least a proportion \((1-\beta)^2\) of ordered pairs satisfy
\(\omega(v_j,v_\ell)\ne0\).  Applying
Lemma~\ref{lem:heis-average-projections} with
\(\delta=(1-\beta)^2\) and \(\alpha=p^{-1/2}\), we get
\begin{equation} \label{IneqRepBound2}
    \|A_\lambda\|_{\operatorname{op}}
    \le1-\frac{(1-\beta)^2}{2}(1-p^{-1/2}).
\end{equation}

Third, we pass from the representation bounds \eqref{IneqRepBound1},
\eqref{IneqRepBound2} to bounds on the Markov chain.  Let
\(\Irr(\Heis)\) denote the following fixed choice of representatives:
\[
    \Irr(\Heis)
    =
    \{\chi_\xi:\xi\in W^*\}
    \cup
    \{\rho_\lambda:\lambda\in\Fp^\times\},
\]
where the representations are those in Proposition~\ref{prop:heis-reps}.
Note that, by Proposition~\ref{prop:heis-reps}, \(\Irr(\Heis)\) contains one representative from each unitary
change-of-basis class of irreducible representations of \(\Heis\).

For each representation \(\rho\in\Irr(\Heis)\), let \(V_\rho\) be its representation space and
write
\[
    B_\rho=\widehat\nu(\rho)=\sum_{g\in\Heis}\nu(g)\rho(g).
\]
The measure \(\nu\) is symmetric because each \(u_{\langle a_j\rangle}\) is
uniform on a subgroup, so \(B_\rho\) is self-adjoint.  Define
\[
    \mathcal H_\rho
    =
    \operatorname{span}\{x\mapsto \langle \rho(x)u,v\rangle:
      u,v\in V_\rho\}\subseteq L^2(\Heis).
\]
By Schur orthogonality
for irreducible representation coefficients (see \textit{e.g.} the discussion of the Peter-Weyl theorem in  \cite[Chapter~2]{Diaconis1988}), we have the orthogonal decomposition
\[
    L^2(\Heis)=\bigoplus_{\rho\in\Irr(\Heis)}\mathcal H_\rho.
\]
The
trivial block, corresponding to \(\chi_0\), is the space of constant
functions.

For \(u,v\in V_\rho\), write
\[
    f_{u,v}^{\rho}(x)=\langle \rho(x)u,v\rangle .
\]
Using the definitions of \(R_\nu,L_\nu\) in \eqref{EqLeftRightDef} and of
\(B_\rho\), we have
\[
\begin{aligned}
    R_\nu f_{u,v}^{\rho}(x)
    &=
    \sum_{g\in\Heis}\nu(g)\langle \rho(xg)u,v\rangle  \\
    &=
    \sum_{g\in\Heis}\nu(g)\langle \rho(x)\rho(g)u,v\rangle  \\
    &=
    \langle \rho(x)B_\rho u,v\rangle
    =
    f_{B_\rho u,v}^{\rho}(x),
\end{aligned}
\]
and similarly
\[
L_\nu f_{u,v}^{\rho}(x)   =f_{u,B_\rho^*v}^{\rho}(x).
\]
Thus \(R_\nu\) and \(L_\nu\) preserve \(\mathcal H_\rho\).  Under the
Peter-Weyl identification of \(\mathcal H_\rho\) with the corresponding
coefficient space, the preceding identities say that \(R_\nu\) acts by
\(B_\rho\) on the first coordinate and \(L_\nu\) acts by \(B_\rho^*\) on the
second coordinate.  Hence
\[
    \|R_\nu|_{\mathcal H_\rho}\|_{2\to2}
    =
    \|L_\nu|_{\mathcal H_\rho}\|_{2\to2}
    =
    \|B_\rho\|_{\operatorname{op}}.
\]
Taking averages, this implies
\[
    \left\|K|_{\mathcal H_\rho}\right\|_{2\to2}
    \le \|B_\rho\|_{\operatorname{op}}.
\]
For every nontrivial \(\rho\), the Inequalities \eqref{IneqRepBound1} and \eqref{IneqRepBound2} give
\[
    \|B_\rho\|_{\operatorname{op}}
    \le 1-\gamma_{p,\beta},
\]
where \(\gamma_{p,\beta}\) is defined in \eqref{eq:heis-gamma-beta}.  Since
\(K\) is self-adjoint on \(L^2(\Heis)\), this bound on the orthogonal
complement of the constants implies that the spectral gap of
\(K\) is at least \(\gamma_{p,\beta}\).

\end{proof}

\begin{lemma}\label{lem:heis-good-fibre-gap}
For all sufficiently large \(r\), the following holds.  Fix \(i\) and freeze the
coordinates \((g_j)_{j\ne i}\).  If the corresponding \(i\)-fibre intersects
\(\Good_{\mathrm H}\), then \(K_i\) in \eqref{eq:heis-fibre-kernel} has spectral
gap at least \(\gamma_{p,\beta_1}\), where \(\beta_1\) is defined in
\eqref{eq:heis-beta0-beta1}.
\end{lemma}

\begin{proof}
Choose a point \(g\in\Good_{\mathrm H}\) in the fibre.  For every nonzero
\(\xi\in W^*\),
\[
    |\{j\ne i:\xi(v_j)=0\}|
    \le |\{j\in[r]:\xi(v_j)=0\}|
    \le\beta_0r.
\]
Thus, for all sufficiently large \(r\),
\[
    \frac1{r-1}|\{j\ne i:\xi(v_j)=0\}|\le\beta_1.
\]
This is precisely the condition
\eqref{eq:heis-hyperplane-balance-fibre} with \(\beta=\beta_1\).  Applying Proposition
\ref{prop:heis-fibre-gap} with this bound proves the claim.
\end{proof}

\subsection{Completing the proof of the upper bound}\label{subsec:heis-apply}

\begin{theorem}[Continuous-time Heisenberg mixing]\label{thm:heis-continuous-time}
Fix an odd prime \(p\) and \(\varepsilon\in(0,1)\).  There is a constant
\(C_{p,\varepsilon}<\infty\) such that, for all sufficiently large \(r\) and all
\(h=2m\ge2\) satisfying \(h\le\varepsilon r\),
\[
    \max_{x\in\Omega}
    \left\|\delta_x e^{C_{p,\varepsilon}r(h+\log r)\log r(P-I)}-\pi\right\|_{\TV}
    \le\frac14.
\]
\end{theorem}

\begin{proof}
We check the conditions of Theorem~\ref{thm:abstract-good-set} with the choices
\[
    S=\Heis,
    \qquad N=r,
    \qquad \Omega=V_r(\Heis),
    \qquad G=\Good_{\mathrm H}.
\]
For all sufficiently large \(r\), the assumption \(h\le\varepsilon r\) implies
\(r\ge h+1\), so Lemma~\ref{lem:heis-connectivity} gives irreducibility on
\(\Omega\) and \(\pi\) is the unique stationary law.  By
Lemma~\ref{lem:heis-good-fibre-gap} and Proposition~\ref{prop:abstract-fibre-gap-lsi},
for every nonnegative \(F:\Heis^r\to\mathbb R\) supported on \(G\),
\[
    \Ent_\mu(F^2)
    \le C_{p,\varepsilon}r(h+1)\cE_P^\mu(F,F),
\]
because \(|\Heis|=p^{h+1}\) and \(\gamma_{p,\beta_1}>0\) depends only on
\(p\) and \(\varepsilon\).  Thus \(A=C_{p,\varepsilon}r(h+1)\).

Lemma~\ref{lem:heis-good-stationary} gives
\[
    \pi(G^c)=o(1),
    \qquad
    A\frac{\pi(G^c)}{\pi(G)}
    \le C_{p,\varepsilon}r(h+1)e^{-c_{p,\varepsilon}r}=o(1),
\]
using \(h\le\varepsilon r\).  Take \(L=r^3\).  With \(M=10\), Lemma~\ref{lem:heis-burnin} gives
\[
    t_*=O_{p,\varepsilon}(r\log r\,(h+\log r))
\]
and the sojourn estimate \eqref{eq:heis-sojourn}; in the notation of
Theorem~\ref{thm:abstract-good-set}, this means \(\eta=o(1)\).  The confinement
time satisfies
\[
    t_{\mathrm{conf}}
    \le C_{p,\varepsilon}r(h+1)\log(e+\log|\Omega|)
    \le C_{p,\varepsilon}r(h+1)\log(r(h+2)).
\]
Since \(h\le\varepsilon r\), this is
\[
    O_{p,\varepsilon}(r(h+\log r)\log r)
\]
and is \(o(r^3)\).  Therefore
\(\zeta=\Prob(\operatorname{Poi}(t_{\mathrm{conf}})>r^3)=o(1)\).  The right-hand
side of \eqref{eq:abstract-after-burnin} is \(o(1)\), and the Poisson lower tail
in \eqref{eq:abstract-continuous-condition} is also \(o(1)\).  Theorem~\ref{thm:abstract-good-set}
gives the stated continuous-time bound.
\end{proof}

\begin{proof}[Proof of Theorem~\ref{thm:heis-main}]
Theorem~\ref{thm:heis-continuous-time} gives the continuous-time bound for
\(e^{t(P-I)}\).  The standard comparison between continuous time and the
half-lazy kernel gives the same bound up to a universal multiplicative constant
\cite[Theorem~20.3(ii)]{levin2017markov}.
\end{proof}

\subsection{Lower bounds}\label{subsec:heis-lower}

\begin{proof}[Proof of Theorem~\ref{thm:heis-lower}]
We prove the two lower bounds separately.

First, we prove the \(r\log r\) support-spreading bound.  Fix a symplectic
basis \(b_1,\ldots,b_h\) of \(W\) with \(\omega(b_1,b_2)=1\), and start the
chain from the canonical tuple appearing in Lemma~\ref{lem:heis-connectivity}:
\[
    x_0=((b_1,0),\ldots,(b_h,0),(0,1),(0,0),\ldots,(0,0)).
\]
Choose a nonzero \(\xi\in W^*\) with \(\xi(b_1)=1\) and
\(\xi(b_i)=0\) for \(2\le i\le h\).  For the horizontal projection
\[
    Y_t=(\xi(v_1(t)),\ldots,\xi(v_r(t)))\in\Fp^r\setminus\{0\},
\]
the initial support size is one.  Let
\[
    S_t=|\{i:Y_t(i)\ne0\}|.
\]
For the non-lazy chain, using the birth probability in
\eqref{eq:heis-support-birth-death},
\[
    \E[S_{t+1}\mid S_t=s]
    \le s+B_s\le s+\frac{s}{r}.
\]
Hence \(\E S_t\le(1+1/r)^t\le e^{t/r}\).  The half-lazy chain only slows this
support growth, so the same upper bound holds for it.  Let
\(\alpha=\frac12(p-1)/p\).  If \(t=c r\log r\) with \(c<1\), then
\[
    \Prob_{x_0}(S_t\ge\alpha r)
    \le \frac{e^{t/r}}{\alpha r}
    =\frac{r^c}{\alpha r}=o(1).
\]
Thus \(\{S_t<\alpha r\}\) has probability \(1-o(1)\) from this start.

Under stationarity, the projection \(Y=(\xi(v_1),\ldots,\xi(v_r))\) is uniform on
\(\Fp^r\setminus\{0\}\), as in the proof of Lemma~\ref{lem:heis-good-stationary}.
Therefore \(S\) has the law of a \(\operatorname{Bin}(r,(p-1)/p)\) variable
conditioned not to be zero.  Hoeffding's inequality gives
\[
    \pi(S<\alpha r)=e^{-c'_p r+o(r)}=o(1).
\]
The total variation distance at time \(c r\log r\) is therefore \(1-o(1)\), for
any fixed \(c<1\) small enough.  This proves
\(\tau_{\mathrm{mix}}\ge c_p r\log r\), after decreasing \(c_p\) if necessary.

Second, we prove the state-space counting lower bound.  The state space has
\[
    |\Omega|=p^r\prod_{q=0}^{h-1}(p^r-p^q).
\]
Indeed, the central coordinates are arbitrary, and the horizontal coordinates
are spanning \(r\)-tuples in an \(h\)-dimensional vector space.  Since
\(r\ge h+1\),
\[
    \prod_{q=0}^{h-1}(1-p^{q-r})
    \ge \prod_{j=2}^{\infty}(1-p^{-j})=:c_p^{(0)}>0.
\]
Consequently
\begin{equation}\label{eq:heis-state-space-lower}
    |\Omega|\ge c_p^{(0)}p^{r(h+1)}.
\end{equation}

From any fixed state, one step of the half-lazy chain has at most
\[
    M_r=1+2pr(r-1)
\]
possible outcomes.  Therefore, after \(t\) steps from a fixed starting point,
the law is supported on a set \(A_t\subseteq\Omega\) with
\[
    |A_t|\le M_r^t.
\]
Since \(\pi\) is uniform on \(\Omega\),
\[
    \pi(A_t)\le \frac{M_r^t}{|\Omega|}.
\]
If
\[
    t\le c'_p\frac{r(h+1)}{\log(e r)}
\]
with \(c'_p>0\) small enough, then \eqref{eq:heis-state-space-lower} and
\(\log M_r=O_p(\log(e r))\) imply \(\pi(A_t)\le1/2\) for all sufficiently large
\(r\).  Since the chain is supported on \(A_t\) at time \(t\),
\[
    \|\delta_xQ^t-\pi\|_{\TV}\ge1-\pi(A_t)\ge\frac12.
\]
This proves
\[
    \tau_{\mathrm{mix}}\ge c'_p\frac{r(h+1)}{\log(e r)}.
\]
Combining the two estimates and decreasing the constant gives the asserted
maximum lower bound.
\end{proof}

\section*{Acknowledgments}
AS thanks NSERC for support via grant RGPIN-2022-03012. This draft was created in collaboration with GPT Pro 5.5. In particular, it alerted us to \cite{Misaghian2010}. It was also used throughout the editing process. All statements and proofs remain the responsibility of the authors.

\bibliographystyle{plain}
\bibliography{refs_trans}

@article{BenHamou2025,
  author  = {Ben-Hamou, Anna},
  title   = {Mixing time of a matrix random walk generated by elementary transvections},
  journal = {Electronic Communications in Probability},
  volume  = {30},
  year    = {2025},
  pages   = {1--11},
  note    = {Article no. 78}
}

@article{BenHamouPeres2018,
  author  = {Ben-Hamou, Anna and Peres, Yuval},
  title   = {Cutoff for a stratified random walk on the hypercube},
  journal = {Electronic Communications in Probability},
  volume  = {23},
  number  = {31},
  year    = {2018},
  pages   = {1--10}
}

@article{CellerLeedhamGreenMurrayNiemeyerOBrien1995,
  author  = {Celler, Frank and Leedham-Green, Charles R. and Murray, Scott H. and Niemeyer, Alice C. and O'Brien, Eamonn A.},
  title   = {Generating random elements of a finite group},
  journal = {Communications in Algebra},
  volume  = {23},
  number  = {13},
  year    = {1995},
  pages   = {4931--4948}
}

@article{ChungGraham1997,
  author  = {Chung, Fan R. K. and Graham, Ronald L.},
  title   = {Stratified random walks on the {$n$}-cube},
  journal = {Random Structures \& Algorithms},
  volume  = {11},
  number  = {3},
  year    = {1997},
  pages   = {199--222}
}

@book{Diaconis1988,
  author    = {Diaconis, Persi},
  title     = {Group Representations in Probability and Statistics},
  series    = {Institute of Mathematical Statistics Lecture Notes--Monograph Series},
  volume    = {11},
  publisher = {Institute of Mathematical Statistics},
  address   = {Hayward, CA},
  year      = {1988}
}

@article{Diaconis1996WalksOG,
  author  = {Diaconis, Persi and Saloff-Coste, Laurent},
  title   = {Walks on generating sets of {A}belian groups},
  journal = {Probability Theory and Related Fields},
  volume  = {105},
  number  = {3},
  year    = {1996},
  pages   = {393--421}
}

@article{DiaconisHolmesNeal2000,
  author  = {Diaconis, Persi and Holmes, Susan and Neal, Radford M.},
  title   = {Analysis of a nonreversible {Markov} chain sampler},
  journal = {The Annals of Applied Probability},
  volume  = {10},
  number  = {3},
  year    = {2000},
  pages   = {726--752}
}

@article{DiaconisSaloffCoste1996LSI,
  author  = {Diaconis, Persi and Saloff-Coste, Laurent},
  title   = {Logarithmic {Sobolev} inequalities for finite {Markov} chains},
  journal = {The Annals of Applied Probability},
  volume  = {6},
  number  = {3},
  year    = {1996},
  pages   = {695--750}
}

@article{DiaconisShahshahani1981,
  author  = {Diaconis, Persi and Shahshahani, Mehrdad},
  title   = {Generating a random permutation with random transpositions},
  journal = {Zeitschrift f\"ur Wahrscheinlichkeitstheorie und Verwandte Gebiete},
  volume  = {57},
  year    = {1981},
  pages   = {159--179}
}

@article{Hildebrand1992,
  author  = {Hildebrand, Martin},
  title   = {Generating random elements in {$SL_n(F_q)$} by random transvections},
  journal = {Journal of Algebraic Combinatorics},
  volume  = {1},
  number  = {2},
  year    = {1992},
  pages   = {133--150}
}

@article{Kassabov2005,
  author  = {Kassabov, Martin},
  title   = {Kazhdan constants for {$SL_n(\mathbb Z)$}},
  journal = {International Journal of Algebra and Computation},
  volume  = {15},
  number  = {5--6},
  year    = {2005},
  pages   = {971--995}
}

@book{levin2017markov,
  author    = {Levin, David A. and Peres, Yuval and Wilmer, Elizabeth L.},
  title     = {Markov Chains and Mixing Times},
  edition   = {2},
  publisher = {American Mathematical Society},
  address   = {Providence, RI},
  year      = {2017}
}

@article{Misaghian2010,
  author  = {Misaghian, Manouchehr},
  title   = {The representations of the {H}eisenberg group over a finite field},
  journal = {Armenian Journal of Mathematics},
  volume  = {3},
  number  = {4},
  year    = {2011},
  pages   = {162--173}
}

@book{Sturmfels2008Algorithms,
  author    = {Bernd Sturmfels},
  title     = {Algorithms in Invariant Theory},
  edition   = {2},
  series    = {Texts and Monographs in Symbolic Computation},
  publisher = {Springer},
  address   = {Vienna},
  year      = {2008},
  isbn      = {978-3-211-77416-9},
  doi       = {10.1007/978-3-211-77417-6}
}

@incollection{Pak2001,
  author    = {Pak, Igor},
  title     = {What do we know about the product replacement algorithm?},
  booktitle = {Groups and Computation III},
  publisher = {de Gruyter},
  address   = {Berlin},
  year      = {2001},
  pages     = {301--347}
}

@article{PeresTanakaZhai2020,
  author  = {Peres, Yuval and Tanaka, Ryokichi and Zhai, Alex},
  title   = {Cutoff for product replacement on finite groups},
  journal = {Probability Theory and Related Fields},
  volume  = {177},
  number  = {3--4},
  year    = {2020},
  pages   = {823--853}
}

@incollection{SaloffCoste2004,
  author    = {Saloff-Coste, Laurent},
  title     = {Random walks on finite groups},
  booktitle = {Probability on Discrete Structures},
  series    = {Encyclopaedia of Mathematical Sciences},
  volume    = {110},
  publisher = {Springer},
  address   = {Berlin},
  year      = {2004},
  pages     = {263--346}
}

@mastersthesis{sotirakitrap,
  author = {Sotiraki, Aikaterini},
  title  = {Authentication protocol using trapdoored matrices},
  school = {Massachusetts Institute of Technology},
  year   = {2016},
  note   = {S.M. thesis, Department of Electrical Engineering and Computer Science}
}

@article{Halmos1969TwoSubspaces,
  author  = {Halmos, P. R.},
  title   = {Two subspaces},
  journal = {Transactions of the American Mathematical Society},
  volume  = {144},
  year    = {1969},
  pages   = {381--389},
  doi     = {10.2307/1995288}
}

@article{PeresWinkler2013Censoring,
  author  = {Peres, Yuval and Winkler, Peter},
  title   = {Can extra updates delay mixing?},
  journal = {Communications in Mathematical Physics},
  volume  = {323},
  number  = {3},
  year    = {2013},
  pages   = {1007--1016},
  doi     = {10.1007/s00220-013-1776-0}
}

@article{FillKahn2013Comparison,
  author  = {Fill, James Allen and Kahn, Jonas},
  title   = {Comparison inequalities and fastest-mixing {Markov} chains},
  journal = {The Annals of Applied Probability},
  volume  = {23},
  number  = {5},
  year    = {2013},
  pages   = {1778--1816},
  doi     = {10.1214/12-AAP886}
}

@article{Rosenthal1995Minorization,
  author  = {Rosenthal, Jeffrey S.},
  title   = {Minorization conditions and convergence rates for {Markov} chain {Monte Carlo}},
  journal = {Journal of the American Statistical Association},
  volume  = {90},
  number  = {430},
  year    = {1995},
  pages   = {558--566},
  doi     = {10.1080/01621459.1995.10476548}
}

@article{chen2020fast,
  title={Fast mixing of Metropolized {H}amiltonian {M}onte {C}arlo: Benefits of multi-step gradients},
  author={Chen, Yuansi and Dwivedi, Raaz and Wainwright, Martin J and Yu, Bin},
  journal={Journal of Machine Learning Research},
  volume={21},
  number={92},
  pages={1--72},
  year={2020}
}

@article{CaputoMenzTetali2015ApproxTensor,
  author  = {Caputo, Pietro and Menz, Georg and Tetali, Prasad},
  title   = {Approximate tensorization of entropy at high temperature},
  journal = {Annales de la Facult{\'e} des sciences de Toulouse : Math{\'e}matiques},
  series  = {6},
  volume  = {24},
  number  = {4},
  pages   = {691--716},
  year    = {2015},
  doi     = {10.5802/afst.1460}
}

@article{CaputoParisi2021BlockFactorization,
  author  = {Caputo, Pietro and Parisi, Daniel},
  title   = {Block factorization of the relative entropy via spatial mixing},
  journal = {Communications in Mathematical Physics},
  volume  = {388},
  number  = {2},
  pages   = {793--818},
  year    = {2021},
  doi     = {10.1007/s00220-021-04237-1}
}

@article{BlancaCaputoChenParisiStefankovicVigoda2022Mixing,
  author  = {Blanca, Antonio and Caputo, Pietro and Chen, Zongchen and Parisi, Daniel and {\v{S}}tefankovi{\v{c}}, Daniel and Vigoda, Eric},
  title   = {On mixing of {Markov} chains: coupling, spectral independence, and entropy factorization},
  journal = {Electronic Journal of Probability},
  volume  = {27},
  pages   = {1--42},
  year    = {2022},
  doi     = {10.1214/22-EJP867}
}

@article{Cesi2001QuasiFactorization,
  author  = {Cesi, Filippo},
  title   = {Quasi-factorization of the entropy and logarithmic {Sobolev} inequalities for {Gibbs} random fields},
  journal = {Probability Theory and Related Fields},
  volume  = {120},
  pages   = {569--584},
  year    = {2001},
  doi     = {10.1007/PL00008792}
}

@article{LuYau1993KawasakiGlauber,
  author  = {Lu, Sheng Lin and Yau, Horng-Tzer},
  title   = {Spectral gap and logarithmic {Sobolev} inequality for {Kawasaki} and {Glauber} dynamics},
  journal = {Communications in Mathematical Physics},
  volume  = {156},
  number  = {2},
  pages   = {399--433},
  year    = {1993},
  doi     = {10.1007/BF02098489}
}

@article{MadrasRandall2002Decomposition,
  author  = {Madras, Neal and Randall, Dana},
  title   = {{Markov} chain decomposition for convergence rate analysis},
  journal = {The Annals of Applied Probability},
  volume  = {12},
  number  = {2},
  pages   = {581--606},
  year    = {2002},
  doi     = {10.1214/aoap/1026915617}
}

@article{JerrumSonTetaliVigoda2004Decomposable,
  author  = {Jerrum, Mark and Son, Jung-Bae and Tetali, Prasad and Vigoda, Eric},
  title   = {Elementary bounds on {Poincar{\'e}} and log-{Sobolev} constants for decomposable {Markov} chains},
  journal = {The Annals of Applied Probability},
  volume  = {14},
  number  = {4},
  pages   = {1741--1765},
  year    = {2004},
  doi     = {10.1214/105051604000000639}
}

\end{document}